\documentclass[10pt]{article} 
\usepackage{t1enc,amsfonts,latexsym,amsmath,amsthm,verbatim,amssymb,graphics,mathdots,pstool,color}
\usepackage[utf8]{inputenc} 
\usepackage{pgfplots}
\usepackage[showonlyrefs]{mathtools}
\usepackage{courier}
\usepackage{listings}
\usepackage[framed,numbered,autolinebreaks,useliterate]{mcode}
\usepackage{authblk}

\usepackage[margin=1in,a4paper]{geometry} 
\usepackage{ amsmath,amsfonts}
\usepackage{graphicx} 

\usepackage{mathrsfs}
\providecommand{\keywords}[1]{\textbf{\textit{Keywords: }} #1}
\providecommand{\MSC}[1]{\textbf{\textit{Mathematics Subject Classification: }} #1}

\newcommand{\rank}{\mathsf{rank}}
\newcommand{\diag}{\mathsf{diag}}
\renewcommand{\Re}[1]{\mathsf{Re}(#1)}

\DeclareMathOperator*{\argmin}{arg\,min ~~}

\newcommand{\fro}[1]{\left\| #1\right\|^2}
\newcommand{\scal}[1]{\left \langle #1 \right \rangle}
\newcommand{\para}[1]{\left( #1\right)}
\newcommand{\R}{\mathcal{R}}

\newcommand{\bb}[1]{\boldsymbol{#1}}

\newtheorem{theorem}{Theorem}[section]

\newtheorem{lemma}[theorem]{Lemma}
\newtheorem{corollary}[theorem]{Corollary}
\newtheorem{definition}  [theorem] {Definition}

\newcommand{\Hankelset}{\mathcal{H}}
\newcommand{\dualproj}{\mathfrak{B}}

\newcommand{\runS}{s}
\newcommand{\runSh}{s^\dagger}
\newcommand{\U}{\mathsf{U}}

\newcommand{\soc}{\mathfrak{S}}

\def\MM{{\mathfrak{M}}}
\def\M{{\mathcal{M}}}
\newcommand{\I}{\mathscr{J}}
\newcommand{\J}{\mathscr{K}}

\renewcommand{\S}{\mathscr{T}}
\renewcommand{\R}{\mathscr{R}}
\newcommand{\G}{\mathscr{S}}

\title{Fixed-point algorithms for frequency estimation and structured low rank approximation}
\author{Fredrik Andersson} 
\author{Marcus Carlsson} 
\affil{Center for Mathematical Sciences, Lund University, Box 118, 22100 Lund, Sweden}
\date{} 

\begin{document}
\maketitle

\MSC{15B05,  65K10,  41A29,  41A63. }

\begin{abstract}
We develop fixed-point algorithms for the approximation of structured matrices with rank penalties. In particular we use these fixed-point algorithms for making approximations by sums of exponentials, or frequency estimation. For the basic formulation of the fixed-point algorithm we show that it converges to the minimum of  the convex envelope of the original objective function along with its structured matrix constraint. It often happens that this solution agrees with the solution to the original minimization problem, and we provide a simple criterium for when this is true. We also provide more general fixed-point algorithms that can be used to treat the problems of making weighted approximations by sums of exponentials given equally or unequally spaced sampling. We apply the method to the case of missing data, although optimal convergence of the fixed-point algorithm is not guaranteed in this case. However, it turns out that the method often gives perfect reconstruction (up to machine precision) in such cases. We also discuss multidimensional extensions, and illustrate how the proposed algorithms can be used to recover sums of exponentials in several variables, but when samples are available only along a curve.
\end{abstract}

\section{Introduction}

We consider the non-convex problem \begin{equation}\label{probrank}
\argmin_{A\in\Hankelset} \I_{F,\sigma_0}(A) = \argmin_{A\in\Hankelset} \sigma_0^2\rank(A) +\|A-F\|^2,
\end{equation}
where $A$ is an $M\times N$-matrix, $\Hankelset$ any linear subspace, $\sigma_0$ is a parameter and the norm refers to the Frobenius norm. The convex envelope of $\I_{F,\sigma_0}$  is then given by
\begin{equation} \label{probmod}
\R_{\sigma_0}(A) + \|A-F\|^2,
\end{equation}
where
$$
\R_{\sigma_0}(A) = \sum_j \sigma_0^2-\para{\max_{}\para{\sigma_0-\sigma_j(A),0}}^2,
$$
see \cite{larsson2015convex, larsson2014rank} for details.
We provide a fixed point algorithm which is guaranteed to solve the problem
\begin{equation}\label{probLagrange}
\argmin_{A\in \Hankelset} \R_{\tau} (A) +q \|A-F\|^2,
\end{equation}
for any $1<q<\infty$ and $\tau>0$. It turns out that the solution to \eqref{probLagrange} often coincides with the solution to the non-convex problem \eqref{probrank} for $\sigma_0=\tau/\sqrt{q}$, and we provide a simple condition to verify if this has happened. To be more precise, below is a compilation of Theorem \ref{buli} and Theorem \ref{thm:A_structure} in the particular case $q=2$ for the non-linear map $\dualproj_{F,\tau,q}$ defined by (\ref{dualproj_def}).


\begin{theorem}\label{introbuli}
Given any starting point $W^1,$ the Picard iteration $W^{n+1}= \dualproj_{F,\tau,2}(W^n)$ converges to a fixed point $W^{\circ}$.
Moreover, the orthogonal projection $\mathcal{P}_\Hankelset (W^{\circ})$ is unique and $$
A^{\circ}=2F-\mathcal{P}_\Hankelset (W^{\circ}),
$$
is the unique solution to \eqref{probLagrange}. Moreover, $W^{\circ}$ has the property that $A^{\circ}$ also solves \begin{align*}
&\argmin_{A}   \R_{\tau}\para{A} + \fro{ A-W^\circ}.
\end{align*}
and if $\sigma_j(W^\circ)$ has no singular values equal to $\tau$, then $A^{\circ}$ is the solution to the non-convex problem \eqref{probrank} with $\sigma_0=\tau/\sqrt{2}$.
\end{theorem}

To provide more flexibility in choosing the objective functional, we also consider the problem of minimizing
\begin{equation}\label{probequ}
\R_{\tau}(A) + q \fro{\MM A-h}_\mathcal{V},\quad A\in\Hankelset,
\end{equation}
where
$\mathfrak{M}:\Hankelset\rightarrow \mathcal{V}$ is a linear operator into any given linear space. We give conditions under which the fixed points of a certain non-linear operator coincide with the global minimum of this problem.

We apply these algorithms to the problem of approximation by sparse sums of exponentials. 
Let $f:\mathbb{R}\rightarrow \mathbb{C}$ be the function that we wish to approximate, and assume that it is sampled at points $x=x_j$. The problem can then be phrased as
\begin{equation}\label{Hankel_ell2_org}
\argmin_{\{c_k\},\{\zeta_k\}, K} \sigma_0^2 K + \sum_{j=1}^J \mu_j \left| \sum_{k=1}^K c_k e^{ 2 \pi i x_j \zeta_k} - f(x_j) \right|^2, \quad c_k,\zeta_k \in \mathbb{C}, \quad \mu_j \in \mathbb{R}^+.
\end{equation}
Here, the parameter $\sigma_0$ is a parameter that penalizes the number of exponential functions. The weights $\mu_j$ can be used to control the confidence levels in the samples $f(x_j)$.

This approximation problem is closely related to the problem of frequency estimation, i.e., the detection of the (complex) frequencies $\zeta_k$ above. Once the frequencies are known, the remaining approximation problem becomes linear and easy to solve by means of least squares. The literature on frequency estimation is vast. The first approach was made already in 1795 \cite{prony}. Notable methods that are commonly used are ESPRIT \cite{ESPRIT}, MUSIC \cite{MUSIC} and matrix pencil methods \cite{pencil}. We refer to \cite{stoica2005spectral} for an overview from a signal processing perspective. From an analysis point of view, the AAK-theory developed in \cite{adamjan1971analytic,adamyan1968infiniteI,adamyan1968infiniteR} deals with closely related problems, and approximation algorithms based on these results were developed in \cite{andersson2011sparse,beylkin2005approximation}. A competitive algorithm in the case of fixed $K$ was recently developed in \cite{actwIEEE}.

To approach this problem, we let $\Hankelset$ denote the set of Hankel matrices, i.e., if $A\in \Hankelset$, then there is a generating function $f$ such that $A(j,k)=f(j+k)$. Hankel matrices are thus constant along the anti-diagonals. By the Kronecker theorem there is a connection between a sums of $K$ exponential functions and Hankel matrices with rank $K$, namely that if $f$ is a linear combination of $K$ exponential functions sampled at equally spaced points, then the Hankel matrix generated by these samples will generically have rank $K$. For more details, see e.g. \cite{IEOT,acCar,kronecker,rochberg}.

In the special case with equally spaced sampling, i.e. $x_j=j$, and
\begin{equation}\label{triangle_weight}
\mu_j=\begin{cases}
j & \text{if $1\le j \le N+1$} \\
2N+2-j & \text{if $N+1 < j \le 2 N+1$}
\end{cases},
\end{equation}
the problem (\ref{Hankel_ell2_org}) (with $F\in\Hankelset$ being the Hankel matrix generated from $f$) is equivalent with \eqref{probmod}. The triangular weight above comes from counting the number of elements in the Hankel matrix along the anti-diagonals. Theorem \ref{introbuli} thus provide approximate solutions to this problem, which often turn out to solve the original problem \eqref{Hankel_ell2_org} as well, which in this setting is the equivalent of \eqref{probrank}.

Clearly, for many purposes it would be more natural to use the regular $\ell^2$ norm in the approximation, i.e., to use constant values of $\mu_j$. Also more general setups of weights can be of interest (for instance in the case of missing data). We show how the fixed-point algorithm designed for (\ref{probequ}) can be used in this case.


The basic algorithms for frequency estimation (or approximations by sums of exponentials) require that the data is provided (completely) at equally spaced samples. In many situations, this assumption does not hold (in particular concerning sampling in two or more variables). A traditional method designed for computing periodograms for unequally spaced sampling is the Lomb-Scargle method \cite{Lomb1976,scargle1982studies}. The results are typically far from satisfactory, cf. \cite{actwIEEE}. For a review of frequency estimation methods for unequally spaced sampling, see \cite{babu2010spectral}.
In this work, we will use similar ideas as in \cite{actwIEEE} for treating the unequally spaced sampling.  Given a set of sample points $X = \{x_j\}$, let $\mathcal{I}_X$ be a matrix that interpolates between equally spaced points and the points $x_j$. This means for instance that the rows of $\mathcal{I}_X$ have sum one. The action of its adjoint $\mathcal{I}_X^\ast$ is sometimes referred to as \emph{anterpolation}. We study the problem
\begin{equation}\label{ijkll}
\argmin_{A(j,k)=a(j+k)} \sigma_0^2 \rank(A) +  \| ( f-I_X(a)) \|^2.
\end{equation}
using the algorithm for (\ref{probequ}), and we also modify \eqref{ijkll} to deal with the corresponding weighted case.

Finally, we discuss how the methods can be used in the multidimensional case. Let $\Omega\subset \mathbb{R}^d$ be an open bounded connected set and $f$ a function on $\Omega$ which is assumed to be of the form \begin{equation}\label{fgb}f(x)=\sum_{k=1}^K c_ke^{\zeta_k \cdot x},\end{equation}
where $\zeta_k\in\mathbb{C}^d$, $\zeta_k\cdot x=\sum_{j=1}^d\zeta_{k,j}x_j$ and $K$ is a low number. In \cite{IEOT} we introduce a new class of operators called general domain truncated correlation operators, whose ``symbols'' are functions on $\Omega$, and prove a Kroncker-type theorem for these operators. In particular it follows that their rank is $K$ if the symbol is of the form \eqref{fgb}, and the converse situation is also completely clarified. In \cite{acCar} it is explained how these operators can be discretized giving rise to a class of structured matrices called ``general domain Hankel matrices'', (see Section \ref{secmultsfe} for more details). Letting $\Hankelset$ be such a class of structured matrices, it turns out that \eqref{probLagrange} is a natural setting for solving the problem: Given noisy measurements of $f$ at (possibly unequally spaced) points in $\Omega$, find the function $f$. The connected problem of finding the values $\{c_k\}$ and $\{\zeta_k\}$ is addressed e.g. in \cite{ACdeHExp2dI}.

The paper is organized as follows: In Section \ref{sec2} we provide some basic ingredients for the coming sections, in particular we introduce the singular value functional calculus and connected Lipschitz estimates. Our solution of \eqref{probLagrange} involves an intermediate more intricate objective functional, which is introduced in Section \ref{secconvex}, where we also investigate the structure of its Fenchel conjugates. Moreover our solution to \eqref{probLagrange} first solves a dual problem, which is introduced and studied in Section \ref{secdual}. Consequently the key operator $\dualproj_{F,\tau,q}$ is introduced in Section \ref{secdual}. At this point, we have the ingredients necessary for addressing the main problem \eqref{probLagrange}. In Section \ref{secdiscussion} we recapitulate our main objective and discuss how various choices of $q$ and $\tau$ affect the objective functional and also its relation to the original problem \eqref{probrank}. An extended version of Theorem \ref{introbuli} is finally given in Section \ref{secthealg}, where we also present our algorithm for solving \eqref{probLagrange}. The more general version of the algorithm and corresponding theory is given in Section \ref{secgener}. The remainder of the paper is devoted to frequency estimation and numerical examples. Section \ref{seccfe} considers the one-dimensional problem \eqref{ijkll}, whereas the problem of retrieving functions of the form \eqref{fgb} is considered in Section \ref{secmultsfe}. This section also contains a more detailed description of the general domain Hankel matrices which underlies the method. Numerical examples are finally given in Section \ref{secnum} and we end with a summary of our conclusions in Section \ref{secconcl}.

\section{Preliminaries}\label{sec2}
\subsection{Fenchel conjugates and convex optimization}\label{secnew}
We recall that the Fenchel conjugate $f^\ast$ of a function $f$ is given by
$$f^\ast(y)=\max_{x}\quad  \scal{x,y} - f(x).$$
For positive functions $f$, $f^{\ast}$ is always convex, and if $f$ is also convex, then $f=f^{\ast\ast}$, see Theorem 11.1 in \cite{rockafellarwets}. If $f$ is not convex, then $f^{\ast\ast}$ equals the convex envelope of $f$. In particular, note that we always have $f^*=f^{***}$.

Suppose now that $f$ is a convex function on some finite dimensional linear space $\mathcal{V}$, and let $\M$ be a linear subspace. Consider the problem of finding
\begin{equation}\label{Hankelorg2}\begin{aligned}
&\min f(x)\\
&\text{subj. to } x\in \M.
\end{aligned}\end{equation}
If $y\in \M^{\perp}$ then the value in \eqref{Hankelorg2} is clearly larger than or equal to the value
\begin{equation}\label{Dlambda}
  \min_{x} f(x)- \langle y,x\rangle=-f^{\ast}(y).
\end{equation}
Suppose now that we can find a value $y^{\circ}$ (not necessarily unique) which minimizes
\begin{equation}\label{Hankelorg1}\begin{aligned}
&\min f^*(y)\\
&\text{subj. to } y\in \M^\perp.
\end{aligned}\end{equation}
Since $f$ is convex, it is easy to see that the corresponding value in \eqref{Dlambda} equals that of \eqref{Hankelorg2}.

\subsection{The singular value functional calculus}\label{secsingval}
Let $M,N\in\mathbb{N}$ be given and let $\mathbb{M}_{M,N}$ denote the Euclidean space of $M\times N$-matrices, equipped with the Frobenius norm. We will need the singular value functional calculus \cite{andcarper}, defined as follows. Let $f$ be a real valued function on $\mathbb{R}_+$ such that $f(0)=0$ and let $A\in\mathbb{M}_{M,N}$ be a matrix with singular value decomposition $A=U  \diag(\{ \sigma_j(A) \}) V^\ast$. We then define $$\soc_f(A)=U  \diag(\{ f(\sigma_j(A)) \}) V^\ast.$$
The key property we will use is that
\begin{equation}\label{lip}\|\soc_f(A)-\soc_f(B)\|\leq\|f\|_{Lip}\|A-B\|,\end{equation}
where $\|f\|_{Lip}$ refers to the Lipschitz constant for $f$. In other words, Lipschitz functions give rise to Lipschitz functional calculus, with the same bound.

\section{Convex envelopes}\label{secconvex}
For a given matrix $A$, let
$$
\S_{\sigma_0}(A) = \sum_j \max_{}\para{\sigma_j^2(A)-\sigma_0^2,0},
$$
and
$$
\R_{\sigma_0}(A) = \sum_j \sigma_0^2-\para{\max_{}\para{\sigma_0-\sigma_j(A),0}}^2.
$$
For a proof of the following theorem, see \cite{larsson2015convex, larsson2014rank}.

\begin{theorem}\label{kalle} Let $F$ be a fixed matrix, and let
\begin{equation}\label{Ifun0}
\I_{F,\sigma_0}(A) = \sigma_0^2 \rank(A)+ \fro{A-F}.
\end{equation}
The Fenchel conjugates of $\I_{F,\sigma_0}$ are then
\begin{align}
\I_{F,\sigma_0}^\ast (X) &= \S_{\sigma_0}\para{\frac{X}{2}+F} -\fro{F}  \label{Ifun1}, \\
\I_{F,\sigma_0}^{\ast \ast} (A) &= \R_{\sigma_0}(A)+ \fro{A-F} \label{Ifun2}.
\end{align}
\end{theorem}

Let $\tau>0$ and let $1<p,q<\infty$ be conjugate exponents, i.e. such that $1/p+1/q=1$.

\begin{theorem} \label{thm:bi_conjugate}
Let
$$
\J_{F,\tau,q}(A,B)= \tau^2 \rank(A)  + p\fro{(A-B)}+ q\fro{(B-F)}.
$$
It's Fenchel conjugate is then given by
\begin{align} \label{Jast}
&\J_{F,\tau,q}^{\ast}(X,Y)=\S_{\tau}\para{\frac{X}{2}-\frac{Y}{2q}+F}+ {(q-1)}\fro{\frac{Y}{2q}-F}-q \fro{F}.
\end{align}
and its convex envelope (Fenchel bi-conjugate) is given by
\begin{align} \label{Jastast}
\J_{F,\tau,q}^{\ast \ast}(A,B) = \R_{\tau} \para{A} + p \fro{A-B}+ q \fro{B-F},
\end{align}
\end{theorem}
\begin{proof}
We write simply $\J$ when ${F,\tau,q}$ are clear from the context. Consider the Fenchel conjugate of $\J$
\begin{align*}
&\J^\ast(X,Y) = \max_{A,B} \scal{A,X}+\scal{B,Y} - \J(A,B) =\\
& \max_A \scal{A,X} - \tau^2 \rank{A} + \max_B \para{\scal{B,Y} -   p\fro{(A-B)}- q\fro{(B-F)}}.
\end{align*}
The right maximum is attained for
$$
B=\frac{p}{p+q}A + \frac{q}{p+q}F +\frac{1}{2(p+q)}Y.
$$
After substituting $B$ in the Fenchel conjugate and some algebraic simplifications, we obtain
\begin{align}  \label{first_fenchel_XY}
\J^\ast(X,Y) &= \nonumber \max_A \scal{A,X} - \tau^2 \rank{A} - \frac{pq}{p+q}\fro{A- \para{ \frac{Y}{2q}+F}} + q\fro{\frac{Y}{2q}+F}-q \fro{F}  \\
&= \para{ \max_A \scal{A,X} - \tau^2 \rank{A} -\fro{A-\para{\frac{Y}{2q}+F}} }+  q\fro{\frac{Y}{2q}+F}-q \fro{F}.\nonumber
\end{align}
From (\ref{Ifun1}) we see that the expression inside the first parenthesis can be written as
$$
\S_{\tau}\para{\frac{X}{2} + \frac{Y}{2q}+F}-\fro{\frac{Y}{2q}+F},
$$
from which (\ref{Jast}) follows.

Next we study the bi-conjugate
$$
\J^{\ast \ast}(A,B) = \max_{X,Y} \scal{A,X} + \scal{B,Y} - \J^{\ast}(X,Y).
$$
To simplify the computations we change coordinates
\begin{equation}\label{1}
\begin{cases}
C=\frac{X}{2}+\frac{Y}{2q} +F\\
D=\frac{Y}{2q}+F.
\end{cases}
\end{equation}
for which
$$
\begin{cases}
X=2(C-D)\\
Y=2q(D-F).
\end{cases}
$$
The bi-conjugate can then be written as
\begin{align*}
\J^{\ast \ast}(A,B) &= \max_{C,D} 2\scal{A,C-D} + 2q\scal{B,D-F} - \J^{\ast}\para{2(C-D),2q(D-F)}\\
& =\max_{C,D} ~~2\scal{A,C-D} + 2q\scal{B,D-F}- \S_{\tau}\para{C}-(q-1) \fro{D} +q \fro{F}=\\
& =\max_{C}\para{  \scal{A,2C} -\S_{\tau}\para{C}}  + \max_D \para{    2\scal{D,qB-A}-(q-1)\fro{D} }+ q\fro{F}+2q \scal{B,F}.
\end{align*}
The maximization over $D$ is straightforward and it will give a contribution $\frac{\fro{qB-A}}{q-1}$ for $D=\frac{qB-A}{q-1}$. Hence $C=\frac{X}{2}+\frac{qB-A}{q-1}$ by \eqref{1}, and
\begin{align*}
\J^{\ast \ast}(A,B)&= \max_{X}\para{  \scal{A,X}+\frac{2}{q-1}\scal{A,qB-A}  - \S_{\tau}\para{\frac{X}{2}+\frac{qB-A}{q-1}}} \\ &\quad \quad \quad \quad \quad +\frac{\fro{qB-A}}{q-1} + q \fro{F}-2q \scal{B,F} \\  &=\max_{X}\Big(  \scal{A,X}  - \I_{\frac{{qB-A}}{q-1},\tau}^\ast\para{X} \Big)+\frac{q-2}{(q-1)^2}\fro{qB-A}\\ &\quad \quad \quad \quad \quad +\frac{2}{q-1}\scal{A,qB-A} + q \fro{F}-2q \scal{B,F}.
\end{align*}
We recognize the maximization part as a Fencel conjugate, and hence
\begin{align*}
\J^{\ast \ast}(A,B)&=
 \I_{\frac{{qB-A}}{q-1},\tau}^{\ast\ast}\para{A}\!  +\!\frac{q\!-\!2}{(q\!-\!1)^2}\fro{qB-A} \!+\!\frac{2}{q\!-\!1}\scal{A,qB\!-\!A} \!+\! q \fro{F}\!-\!2q \scal{B,F}\\
&=\R_{\tau}(A)\!+\! \Big\| A\!-\!\frac{{qB\!-\!A}}{q\!-\!1} \Big\|  +\frac{q\!-\!2}{(q\!-\!1)^2}\fro{qB\!-\!A}\! +\!\frac{2}{q\!-\!1}\scal{A,qB\!-\!A} \!+\! q \fro{F}\!-\!2q \scal{B,F} \\ &= \R_{\tau}(A) +\frac{q}{q-1}\fro{A-B}+q\fro{B-F}.  \qedhere
\end{align*}
\end{proof}

\section{The dual optimization problem}\label{secdual}
Let $\Hankelset\subset \mathbb{M}_{M,N}$ be any fixed linear subspace.
We consider the problem
\begin{equation}\label{daproblemHankel}\begin{aligned}
&\argmin \J^{\ast \ast}(A,B) = \R_{\tau}(A)+ p \fro{A-B}+q \fro{B-F}\\
&\text{subj. to } A=B, B\in\Hankelset,
\end{aligned}\end{equation}
where $F\in\Hankelset$.
Let $\M$ be the linear subspace of $\mathbb{M}_{M,N}\times \mathbb{M}_{M,N}$ defined by the equations $A=B$ and $B\in\Hankelset$. It is easy to see that $\M^\perp$ is defined by $$(X,Y)\in \M^\perp \Leftrightarrow X+Y\in \Hankelset^{\perp}.$$
Indeed, $\Leftarrow$ is obvious and the other part can be established by noting that $\dim \M=\dim\Hankelset$ whereas the subspace defined by $X+Y\in \Hankelset^{\perp}$ has dimension $\dim \Hankelset^{\perp}+\dim \mathbb{M}_{M,N}$, (which thus equals $\dim\mathbb{M}_{M,N}\times \mathbb{M}_{M,N}-\dim\M$). Motivated by the results in Section \ref{secnew}, we now focus on the ``dual problem'' to \eqref{daproblemHankel}, i.e.
\begin{equation}\label{daproblemHankel2}
\begin{aligned}
&\argmin \J^*(X,Y) \\
&\text{subj. to } X+Y\in\Hankelset^{\perp}
\end{aligned}
\end{equation}
Set 
\begin{equation}\label{runS}
\runS_{\tau,q}(\sigma)= \max\left( \min\left( \sigma,\tau\right),\frac{\sigma}{q}\right).
\end{equation}
When there is no risk for confusion we will omit the subindices. Recall the singular value functional calculus introduced in Section \ref{secsingval}.

\begin{figure}
\begin{center}
\includegraphics[width=0.5\linewidth,trim=6in 5in 6in 5in]{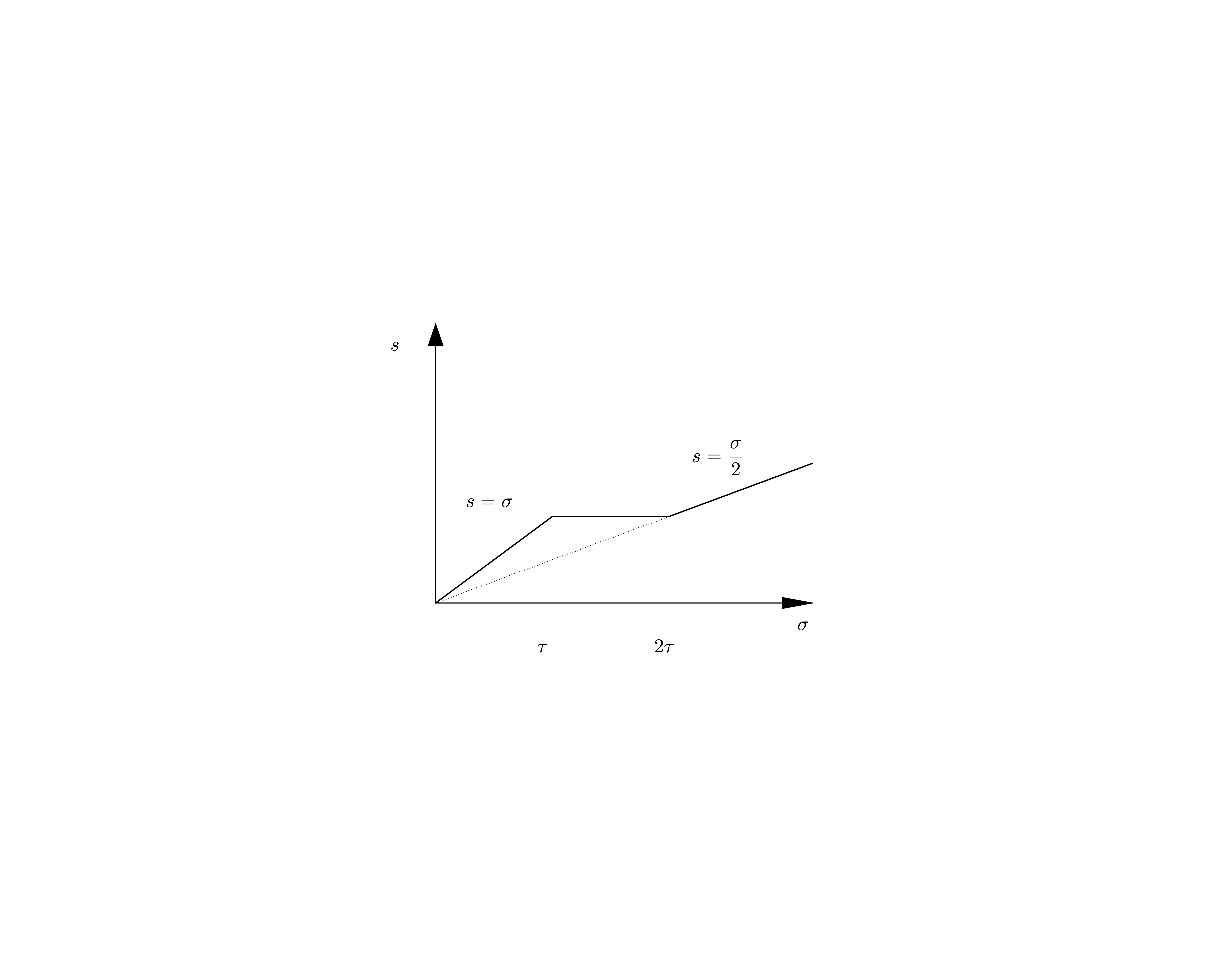}
\end{center}
\caption{An illustration of the function $\runS_{\tau,2}(\sigma)$ used in Lemma \ref{lemma:runS} }
\end{figure}

\begin{lemma} \label{lemma:runS}
$$
\soc_{\runS}(A) = \underset{W}{\arg\min}  \quad  \S_{\tau}\para{W} +  \frac{1}{q-1}\fro{W-A}.
$$
\end{lemma}
\begin{proof}
Since $\S_{\tau}(W)$ depends only on the singular values of $W$, von-Neumann's inequality \cite{von1937some} shows that the solution to the above problem will have the same singular vectors as $A$, (see the proof of Theorem \ref{kalle} in \cite{andcarper} for more details). Henceforth it will be of the form $\soc_\runS(A)$ where $\runS$ is defined by
\begin{equation}\label{muformel2}
\runS(\sigma_j(A))= \min_{\omega} \left(  \max_{}\para{\omega^2-\tau^2,0}+ \frac{1}{q-1}\left( \omega-\sigma_j(A)\right)^2  \right),
\end{equation}
and by equating the subgradient of the above expression with respect to $\omega$ with zero, it is easily verified that this is solved by the function \eqref{runS}.
\end{proof}

Let $\mathcal{P}_{\Hankelset}$ denote the operator of orthogonal projection onto $\Hankelset$.
We now introduce the (non-linear) operator $\dualproj_{F,\tau,q}:\mathbb{M}_{M,N}\rightarrow \mathbb{M}_{M,N}$ which will play a key role in our solution of the original problem \eqref{probLagrange}.
\begin{definition}
Set \begin{equation} \label{dualproj_def}
\dualproj_{F,\tau,q}(W) = \soc_{\runS_{\tau,q}}\para{q F + \mathcal{P}_{\Hankelset^\perp}\para{W}},
\end{equation}
which we abbreviate $\dualproj$ when $F,\tau,q$ are clear from the context.\end{definition}
\begin{theorem}\label{uli} There exists a fixed point $W^{\circ}$ of $\dualproj$ such that the problem
\begin{equation} \label{f_ast_X+Y}
 \begin{aligned}
 & \underset{X,Y}{\text{minimize}}
 & & \J^\ast(X,Y) \\
 & \text{subject to}
 & & X+Y \in \Hankelset^\perp.
 \end{aligned}
\end{equation}
is solved by
\begin{equation}\label{XY-formulas}\begin{aligned}
X^{\circ}&={2p}\para{\mathcal{P}_{\Hankelset}\para{W^{\circ}}-F} +2\mathcal{P}_{\Hankelset^\perp}\para{W^{\circ}} \\
Y^{\circ}&=- {2p}\para{\mathcal{P}_{\Hankelset}\para{W^{\circ}}-F}.
\end{aligned}\end{equation}
\end{theorem}

\begin{proof}
Since $X+Y\in\Hankelset^{\perp}$ we can write
$$
X=2q H+H_1^\perp, \quad Y= -2q H+ H_2^\perp.
$$
with $H \in \Hankelset$ and $H_1^\perp, H_2^\perp \in \Hankelset^\perp$.
By \eqref{Jast} the problem \eqref{f_ast_X+Y} becomes
\begin{align*}
&\underset{H,H_1^\perp, H_2^\perp}{\arg\min} \J^\ast\para{2q H+H_1^\perp, -2q H+ H_2^\perp} =\\
&\underset{H,H_1^\perp, H_2^\perp}{\arg\min}  \S_{\tau} \para{\frac{2q H +H_1^\perp}{2}+\frac{-2q H+ H_2^\perp}{2q}+F } \!+\! (q\!-\!1)\fro{\frac{ -2q H+ H_2^\perp}{2q}+F}.\\
\end{align*}
Note that $$
\frac{2q H+H_1^\perp}{2}+\frac{-2q H+ H_2^\perp}{2q} = (q-1)H +\frac{H_1^\perp}{2}+\frac{H_2^\perp}{2q} =(q-1) H + H^\perp,
$$
where $H^\perp =\frac{H_1^\perp}{2}+\frac{H_2^\perp}{2q}$. The problem can thus be rewritten
$$
\underset{H, H_2^\perp,H^\perp}{\arg\min} \quad \S_{\tau}\para{F+(q-1)H+H^\perp}  + (q-1)\fro{F-H+\frac{ H_2^{\perp}}{2q}},
$$
Since $F,H\in\Hankelset$ and $H_2^\perp\in\Hankelset^{\perp}$, it follows that
$$
\fro{F-H+\frac{H^\perp_2}{2q}} = \fro{F- H }+\fro{\frac{ H_2^\perp}{2q}},
$$
and hence we conclude that the optimal $H_ 2^\perp$ is the zero matrix. Thus $$H^\perp =\frac{H_1^\perp}{2},$$ and the objective functional can be replaced by
$$
\underset{H,H^\perp}{\arg\min} \quad  \S_{\tau}\para{F+(q-1)H+H^\perp}   + (q-1)\fro{H-F}.
$$
Using the substitution $W=F+(q-1)H+H^\perp$ it then becomes
\begin{align*}
&\S_{\tau}\para{W}  +  (q-1)\fro{\frac{W}{q-1}-\para{ \frac{qF+H^\perp}{q-1}}} = \\&\S_{\tau}\para{W}  + \frac{1}{q-1} \fro{W-\para{q F+ \mathcal{P}_{\Hankelset^\perp}\para{W}}}.
\end{align*}
The minimization problem will thus be solved by
\begin{equation}\label{ming1}
W =\underset{W}{\argmin}~  \G(W,P_{\Hankelset^{\perp}}W),
\end{equation}
where $\G:\mathbb{M}_{M,N}\times \Hankelset^{\perp}\rightarrow\mathbb{R}$ is the functional defined by
$$
\G(W,V)=\S_{\tau}\para{W} +\frac{1}{q-1}\fro{W-\para{q F + V}}.
$$
Note that $\G$ is convex since it is obtained from the convex function $\J^\ast$ by affine changes of coordinates. It is also easy to see that $$\lim_{R\rightarrow\infty}\sup_{\|W\|+\|V\|>R} \G(W,V)=\infty, $$ so it has at least one global minimum. Let $(W^{\circ},V^{\circ})$ be one such. Since $\G(W^{\circ},V^{\circ})=\min_{V}\G(W^{\circ},V)$, it is easy to see that $V^{\circ}=P_{\Hankelset^{\perp}}W^{\circ}$. But then $$W^{\circ}=\argmin_W \G(W,V^{\circ})=\argmin_W \G(W,P_{\Hankelset^{\perp}}W^{\circ}),$$
which by Lemma \ref{lemma:runS} read
$$
W^{\circ}=\soc_{\runS}\para{qF + \mathcal{P}_{\Hankelset^\perp}\para{W^{\circ}}}=\dualproj(W^{\circ}).$$
It is also evident that $W^{\circ}$ is a solution to \eqref{ming1}. The formulas in \eqref{XY-formulas} now follows by tracing the changes of variables backwards.
\end{proof}

\section{A fixed point algorithm}\label{secfixed}
\subsection{Discussion of the objective functional}\label{secdiscussion}
To recapitulate our main objective, we wish to minimize \begin{equation} \label{Hankel_org1}
\sigma_0^2 \rank(A) + \|A-F\|^2,\quad A\in \Hankelset,
\end{equation}
which we replace by its convex envelope
\begin{equation} \label{Hankel_1}
\R_{\sigma_0}(A) + \|A-F\|^2,\quad A\in \Hankelset,
\end{equation}
to achieve convexity. We will in this section show how to minimize
\begin{equation}\label{Lagrange_final}
\argmin_{A\in \Hankelset} \R_{\tau} (A) +q \|A-F\|^2,
\end{equation}
for any $1<q<\infty$ and $\tau>0$. Note that the corresponding modification to the original objective functional \eqref{Hankel_org1}, i.e. \begin{equation} \label{new}
\argmin_{A\in \Hankelset}\tau^2 \rank(A) + q\|A-F\|^2,
\end{equation} does not change the problem, since we may equivalently solve \eqref{Hankel_org1} with $\sigma_0=\tau/\sqrt{q}$ (to see this, multiply \eqref{new} with $1/q$). However, this is not the case with \eqref{Lagrange_final}, for it is easy to show that $$\frac{1}{q}\R_{\tau}(A)=\R_{\tau/\sqrt{q}}(A/\sqrt{q}).$$
Thus \eqref{Lagrange_final} is equivalent with \begin{equation}\label{Lagrange_final_variation}
\argmin_{A\in \Hankelset} \R_{\sigma_0} (A/\sqrt{q}) +\|A-F\|^2,
\end{equation}
where $\sigma_0=\tau/\sqrt{q}$, compare with \eqref{Hankel_org1} and \eqref{Hankel_1}. Note as well that
\begin{align*}&\R_{\tau} (A) +q \|A-F\|^2=\R_{\tau} (A) + \|A-F\|^2+(q-1)\fro{A-F},\end{align*} by which we conclude that the objective functional in \eqref{Lagrange_final} is strictly convex with a unique minimizer $A^{\circ}$. Moreover, the above expression is clearly less than or equal to \begin{align*}\tau^2 \rank(A) + \|A-F\|^2+(q-1)\fro{A-F} =\tau^2 \rank(A) + q\|A-F\|^2,\end{align*} so if it happens that
\begin{equation}\label{amotsdal}
\R_{\tau} (A^{\circ}) +q \|A^{\circ}-F\|^2=\tau^2 \rank(A^{\circ}) + q\|A^{\circ}-F\|^2,
\end{equation}
then $A^{\circ}$ is a solution to the original problem \eqref{new}. However, by the definition of $\R_{\tau}$ it follows that \eqref{amotsdal} is satisfied if and only if $A^{\circ}$ has no singular values in the interval $(0,\tau)$.

\subsection{The basic fixed-point algorithm}\label{secthealg}

We now solve \eqref{Lagrange_final}, using fixed points of the operator $\dualproj_{F,\tau,q}$ from the previous section. It is easy to see that these may not be unique. Despite that, we have the following.

\begin{table}
\begin{lstlisting}[language=Matlab, basicstyle=\scriptsize, breaklines=false]
function A=fixedpoint(F,tau_0, N_iter),
  W=0*F;
  for n=1:N_iter,
    [u,s,v]=svd(2*F+PHp(W));
    W=u*max(min(s,tau_0),s/2)*v';
  end;
  A=2*F-PH(W);
\end{lstlisting}
\begin{lstlisting}[language=Matlab, basicstyle=\scriptsize, breaklines=false]
Lambda=@(a)hankel(a(1:(end+1)/2),a((end+1)/2:end));
function a=iLambda(A),
  N=size(A,1)-1;for j=-N:N, a(N+1+j,1)=mean(diag(flipud(A),j));end
PH=@(A)Lambda(iLambda(A));PHp=@(A)A-PH(A);
\end{lstlisting}

\caption{\label{tab:alg} MATLAB implementation of the fixed point algorithm of Theorem \ref{buli} for $q=2$, along with subroutines for Hankel operations.}
\end{table}

\begin{theorem}\label{buli}
The Picard iteration $W^{n+1}= \dualproj_{F,\tau,q}(W^n)$ converges to a fixed point $W^{\circ}$.
Moreover, $\mathcal{P}_\Hankelset (W^{\circ})$ is unique and $$
A^{\circ}= \frac{1}{q-1}\left( q F- \mathcal{P}_\Hankelset (W^{\circ}) \right),
$$
is the unique solution to \eqref{Lagrange_final}.
\end{theorem}
Before the presenting the proof, let us give a few more comments on the equivalent problem formulation \eqref{Lagrange_final_variation}. Note that $$\lim_{q\rightarrow 0^+}\R_{\tau} (A/\sqrt{q})=\tau^2\rank(A),$$ but that our method requires $q>1$ (in fact, the objective functional is no longer convex beyond $q=1$). However, choosing $q\approx 1$ will yield slow convergence of the Picard iteration since then $\runS_{\tau,q}(\sigma)\approx \sigma$. On the other hand, the effect of choosing a large $q$ (for a fixed $\tau$) is that the term $\R_{\tau} (A/\sqrt{q})$ deviates further from $\tau^2\rank(A)$. We have found that $q=2$ works well in practice.

\begin{proof}
We first prove that the Picard iteration converges. Note that the sequence $V^n=\mathcal{P}_{\Hankelset^{\perp}}W^n,$ $n\in\mathbb{N}$, is generated by the Picard iteration of the operator $\mathcal{P}_{\Hankelset^{\perp}}\dualproj$ and the starting-point $V^0=\mathcal{P}_{\Hankelset^{\perp}}W^0$. Moreover $W^{n+1}=\dualproj(V^n)$. Since $\dualproj$ is continuous (in fact, it is non-expansive by \eqref{lip}) it suffices to show that $(V^n)_{n=0}^\infty$ is convergent.

It is well known \cite{browder1967convergence,reich1987asymptotic} that for firmly non-expansive operators (in finite dimensional space) the Picard iteration converges to a fixed point as long as the set of fixed points of the operator is non-empty. By Theorem \ref{uli} there exists a fixed point $W^{\circ}$ of $\dualproj$, and it clearly follows that $V^{\circ}=\mathcal{P}_{\Hankelset^{\perp}}W^{\circ}$ is a fixed point of $\mathcal{P}_{\Hankelset^{\perp}}\dualproj$. It remains to show that $\mathcal{P}_{\Hankelset^{\perp}}\dualproj$ is firmly non-expansive, as an operator on $\Hankelset^{\perp}$. An equivalent to firmly non-expansive is that $2\mathcal{P}_{\Hankelset^{\perp}}\dualproj -I$ is non-expansive \cite{bauschke2007fenchel}, i.e.,
\begin{equation}\label{sol}
\left\| \left(2\mathcal{P}_{\Hankelset^{\perp}}\dualproj(V)-V\right)-\left(2\mathcal{P}_{\Hankelset^{\perp}}\dualproj(W)-W\right) \right\| \le \| V-W \|, \quad \forall V,W\in \Hankelset^{\perp},
\end{equation}
which we will verify shortly. To this end, note that on $\Hankelset^{\perp}$ we have    \begin{equation}\label{sol1}2\mathcal{P}_{\Hankelset^{\perp}}\dualproj(V)=2\mathcal{P}_{\Hankelset^{\perp}}\runS(qF+V),\end{equation}
and set
\begin{align*}
\runSh(\sigma) &= 2\runS(\sigma)-\sigma = \max\left( \min\left( 2 \sigma,2 \tau\right),\frac{2}{q}\sigma\right)-\sigma \\&=\max\left( \min\left( \sigma,2 \tau-\sigma \right),\para{\frac{2}{q}-1}\sigma\right).
\end{align*}
This function $\runSh$ is clearly Lipschitz continuous with Lipschitz constant equal to one. Setting $\tilde{V}=(qF+V)$ and $\tilde{W}=(qF+W)$, \eqref{lip} and \eqref{sol1} implies that
\begin{align*}
&\left\| \big(2\mathcal{P}_{\Hankelset^{\perp}}\dualproj(V)-V\big)-\big(2\mathcal{P}_{\Hankelset^{\perp}}\dualproj(W)-W\big) \right\|^2  = \\&
\left\|\mathcal{P}_{\Hankelset^\perp}\left( \big(2  \runS(\tilde V)-V \big)-\big(2  \runS(\tilde W)-W \big)\right) \right\|^2 = \\
&\left\|\mathcal{P}_{\Hankelset^\perp}\left( \big(2  \runS(\tilde V)-\tilde V \big)-\big(2  \runS({\tilde W})-\tilde W \big)\right) \right\|^2 =
\left\|\mathcal{P}_{\Hankelset^\perp}\left( \runSh ({\tilde V})- \runSh ({\tilde W})\right)  \right\|^2 \le \\
&\left\| \runSh ({\tilde V})- \runSh ({\tilde W})  \right\|^2\le
\left\| {\tilde V}- {\tilde W}  \right\|^2=\|V-W\|,\end{align*} which establishes \eqref{sol}, and hence the original Picard iteration converges to a fixed point $W^{\circ}$ of $\dualproj$.
For the uniqueness of $\mathcal{P}_{\Hankelset}(W^\circ)$, suppose that $\tilde{W}^{\circ}$ is another fixed point of $\dualproj$ and write $W^{\circ}=H+T,~\tilde{W}^{\circ}=\tilde{H}+\tilde{T}$, where $H, \tilde{H} \in \mathcal{H}$ and $T, \tilde{T} \in \mathcal{H}^\perp$. Then
\begin{align*}
&\| H-\tilde{ H}\|^2 +\| T- \tilde{T}\|^2 =   \| {W}^{\circ}- \tilde{W}^{\circ}\|^2  = \| \dualproj(W^{\circ})-\dualproj(\tilde{W}^{\circ})\|^2 = \\&\| \runS(qF + T)-\runS(qF + \tilde{T})\|^2\le \| (q F + T)-(q F + \tilde{T})\|^2 = \| T- \tilde T\|^2,
\end{align*}
where the last inequality is due to \eqref{lip} and the fact that $\runS_{\tau}$ has Lipschitz constant equal to one. But then
$$
\| H- \tilde{H}\|^2 \le  \| T- \tilde{T}\|^2 - \| T- \tilde{T}\|^2 =0,
$$
implying that
$
H=\tilde H.
$
Finally, if $(A^{\circ},B^{\circ})$ denotes the solution to (\ref{daproblemHankel}), then it is clear that $A^{\circ}$ solves (\ref{Lagrange_final}). By the theory in Section \ref{secnew} and Theorem \ref{uli}, it follows that $(A^{\circ},B^{\circ})$ also is a solution to
\begin{align} \label{argminAB}
\argmin_{A,B}   & \R_{\tau}\para{A} + p\fro{A-B}+q \fro{B-F} \\&- \left\langle A,2p \para{\mathcal{P}_{\Hankelset}\para{W^{\circ}}-F} +2\mathcal{P}_{\Hankelset^\perp}\para{W^{\circ}}  \right\rangle - \langle B, -2p \para{\mathcal{P}_{\Hankelset}\para{W^{\circ}}-F} \rangle.\nonumber
\end{align}
where $W^{\circ}$ is as above. By fixing $A^{\circ}$ we find that
\begin{equation}\label{fg}
\begin{aligned}
B^{\circ}&=\argmin_{B}     p \fro{A^{\circ}-B}+q \fro{B-F} - \left \langle B, -2p \para{\mathcal{P}_{\Hankelset}\para{W^{\circ}}-F} \right \rangle \\
 &= {F} + \frac{1}{q}(A^{\circ}-\mathcal{P}_{\Hankelset}\para{W^{\circ}}).
\end{aligned}
\end{equation}
But inserting the constraint $A^{\circ}=B^{\circ}$ (from \eqref{daproblemHankel}) in the above equation yields
\begin{equation}\label{A_W_solution}
(q-1)A^{\circ}= q F-\mathcal{P}_{\Hankelset}\para{W^{\circ}},
\end{equation}
as desired. The uniqueness of $A^{\circ}$ was argued before the proof, so it is complete.
\end{proof}

%
%

The matrix $W^{\circ}$ has several peculiar properties, as shown by the next theorem, (compare with \eqref{Lagrange_final} and note the absence of the condition $A\in\Hankelset$ below).

\begin{figure}
\begin{center}
\includegraphics[width=0.8\linewidth,trim=0in 0in 0in 0in]{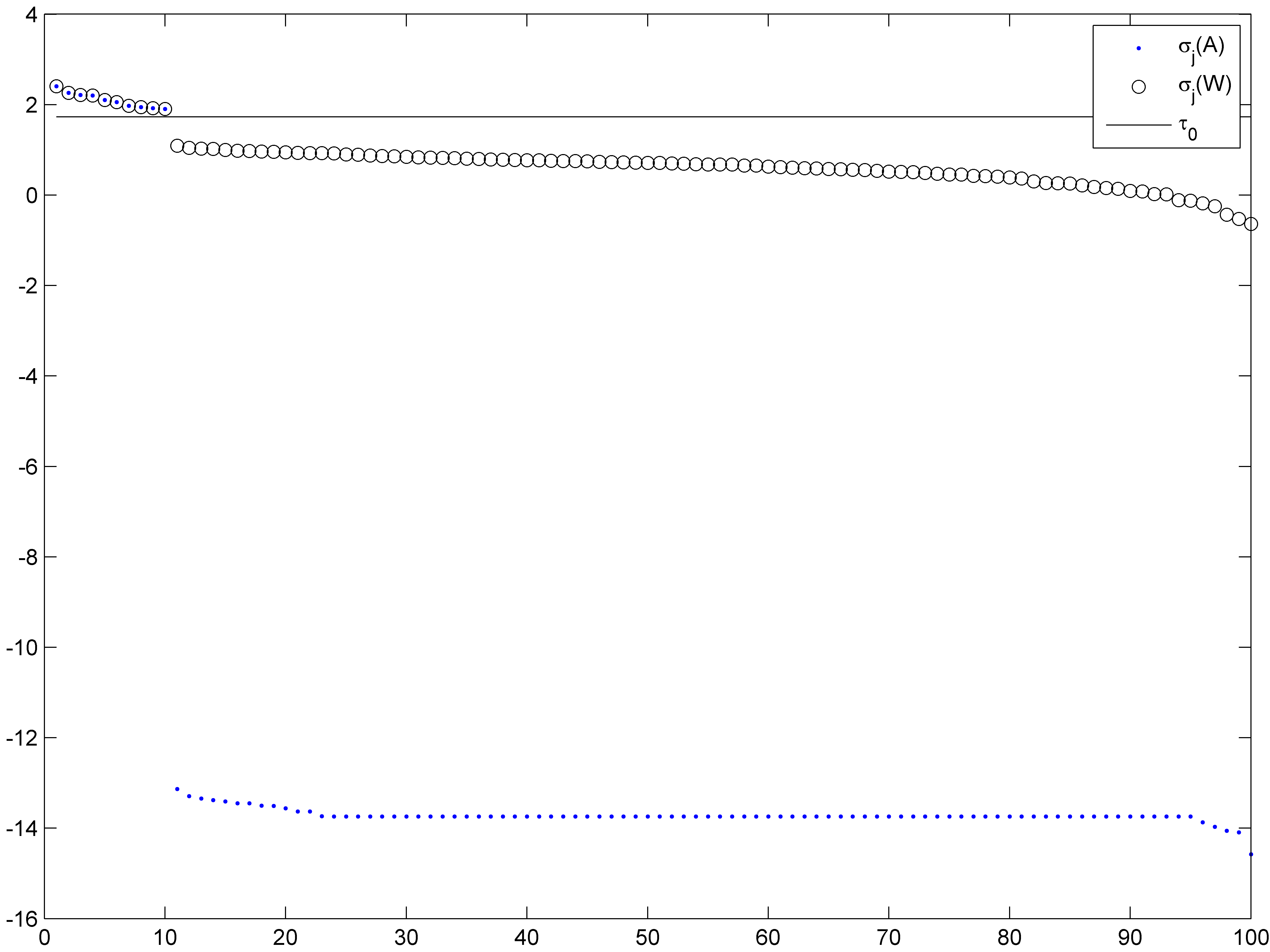}
\end{center}
\caption{Illustration of Theorem \ref{thm:A_structure}. The blue dots show $\sigma_j(A)$,the the black circles $\sigma_j(W)$, and the black line the level of $\tau$. For $j$ such that $\sigma_j(W)>\tau$, it holds that $\sigma_j(W) = \sigma_j(A)$. In this case $\sigma_j(A) \approx 0$ for the larger values of $j$, implying that the obtained result is in fact a solution to the non-convex problem (\ref{new}) as well as a solution to the convex problem (\ref{Lagrange_final_variation}).}
\end{figure}

\begin{theorem} \label{thm:A_structure}
Let $W^{\circ}$ and $A^{\circ}$ be as in Theorem \ref{buli}. Then $A^{\circ}$ solves \begin{align*}
&\argmin_{A}   \R_{\tau}\para{A} + \fro{ A-W^\circ}.
\end{align*}
Moreover, if $W^{\circ} = U \diag(\sigma_j(W^{\circ})) V^\ast$, then $A^{\circ} = U \diag(\sigma_j(A^{\circ})) V^\ast$, where
\begin{equation}\label{sigmaj(A)}
 \begin{cases}\sigma_j(A^{\circ}) =\sigma_j(W^{\circ}), &\text{if $\sigma_j(W^{\circ}) >\tau$}\\
 0\leq \sigma_j(A^{\circ}) \leq \sigma_j(W^{\circ}), &\text{if $\sigma_j(W^{\circ}) =\tau$}\\
\sigma_j(A^{\circ}) =0, &\text{if $\sigma_j(W^{\circ}) <\tau$}
\end{cases}.
\end{equation}
In particular, if $\sigma_j(W^{\circ})$ has no singular values equal to $\tau$, then $A^{\circ}$ is the solution to the non-convex problem \eqref{new}.
\end{theorem}
\begin{proof}

By \eqref{fg}, $A^{\circ}$ minimizes (\ref{argminAB}) restricted to the subspace $B=F + \frac{1}{q}(A-\mathcal{P}_{\Hankelset}\para{W^{\circ}})$. By inserting this expression for $B$ in \eqref{fg} and by a completion of squares, we get
\begin{align*}
A^{\circ}=&\argmin_{A}   \R_{\tau}\para{A} + \fro{ A-W^\circ}.
\end{align*}
Since $\R_{\tau}$ only depend on the singular values of $A$, it follows by von-Neumann's inequality \cite{von1937some} that it is sufficient to consider the minimization of
\begin{equation}\label{AW_sigma_derivation_relation}
\sigma_j(A) =\argmin_{s} -\para{\max_{}\para{\tau-s,0}}^2+\para{s-\sigma_j(W)}^2 .
\end{equation}
The derivative of the function to be minimized is $2(\tau-\sigma_j(W))$ for $s\leq \tau$ and $2(s-\sigma_j(W))$ for $s\geq \tau$, by which \eqref{sigmaj(A)} easily follows. The final statement is immediate by the discussion in Section \ref{secdiscussion}.
\end{proof}

\subsection{The general fixed-point algorithm}\label{secgener}
In this Section we introduce a more general algorithm which allows increased flexibility. Examples of applications will be given in the coming sections. Let $\mathcal{V}$ be an Euclidean space and let $\mathfrak{M}:\Hankelset\rightarrow \mathcal{V}$ be a linear operator. Given $h\in\mathcal{V}$ define
$$
(W',A') = \dualproj_{h,\tau,q}^\MM (W,A),
$$
by
\begin{align*}
W' &= \soc_{\runS_{\tau,q}}(q \left( \MM^* (h-\MM A) + A\right) +\mathcal{P}_{\Hankelset^\perp}(W)),\\
A' &=  \frac{1}{q-1}\left( q\left( \MM^* (h-\MM A) + A\right) - \mathcal{P}_{\Hankelset}(W') \right).
\end{align*}
When there is no risk of confusion we will omit the subindices $h,\tau,q$ from the notation, and we will let $I$ denote the identity operator on $\Hankelset$.
\begin{theorem}\label{thmweighted}
Let $\MM$ be such that $\MM^*\MM\leq I$. Then there exists $W^{\circ}$ such that $(W^{\circ},A^{\circ})$ is a fixed point of $\dualproj_{h,\tau,q}^\MM$ if and only if $A^{\circ}$ is a stationary point of the objective functional
\begin{equation}\label{equ}
\R_{\tau}(A) + q \fro{\MM A-h}_\mathcal{V}
\end{equation}
(as a functional on $\Hankelset$). The objective functional is convex if $q\MM^*\MM\geq I$, and strictly convex if $q\MM^*\MM> I$.
\end{theorem}

\begin{proof}
The statements about convexity follows from the identity
\begin{align*}
&\R_{\tau}(A) + q \fro{\MM A-h}_\mathcal{V}=\\&\R_{\tau}(A) +\fro{A}_{\mathbb{M}_{M,N}}+ \langle (q\MM^*\MM-I)A,A \rangle_{\mathbb{M}_{M,N}}+2q\mathsf{Re} \scal{\MM A,h}_{\mathcal{V}}+\fro{h}_\mathcal{V},
\end{align*}
and the fact that $\R_{\tau}(A)$ is defined so that $\R_{\tau}(A)+ \fro{A}$ becomes convex (but not strictly convex).
A stationary point $A^{\circ}$ of \eqref{equ} is clearly stationary also for
\begin{equation}\label{equ1}
\R_{\tau}(A) +q\| \MM A-h\|^2_\mathcal{V}+q\| \sqrt{I-\MM^*\MM}(A-A^{\circ})\|^2,\quad A\in \Hankelset.
\end{equation}
The objective function can be rewritten
\begin{equation}\label{equ2}
\R_{\tau}(A) + q\|  A-\MM^*h-{(I-\MM^*\MM)}( A^{\circ})\|^2+r(A^{\circ},h),\quad {A \in \Hankelset }.
\end{equation}
where $r$ is a remainder. Note that this is convex in $A$. Theorem \ref{buli} implies that there exists a fixed point $W^{\circ}$ to $\dualproj_{\MM^* (h-{\MM}A^{\circ})+A^{\circ}}$ such that
$$
\frac{1}{q-1}\left( q (\MM^* (h-{\MM}A^{\circ})+A^{\circ})-\mathcal{P}_{\Hankelset}(W^{\circ}) \right)
$$ is a minimizer of \eqref{equ2}. Moreover the theorem says that this minimizer is unique, so there can be no other stationary points to the functional. Hence
$$
A^{\circ}=\frac{1}{q-1}\left( q (\MM^* (h-{\MM}A^{\circ})+A^{\circ})- \mathcal{P}_{\Hankelset}(W^{\circ}) \right),
$$ and thus $(W^{\circ},A^{\circ})$ is a fixed point of $\dualproj_h^{\MM}$.
Conversely, let $(W^{\circ},A^{\circ})$ be a fixed point of $\dualproj_h^{\MM}$. Then, by following the earlier argument backwards, we see that $A^{\circ}$ solves \eqref{equ1}. Since the differential of the last term is zero at $A=A^{\circ}$, the subdifferential of the former two is zero as well, so $A^{\circ}$ is a stationary point of \eqref{equ}.

\end{proof}

\section{Complex frequency estimation}\label{seccfe}
If we let $\Hankelset\subset\mathbb{M}_{M,N}$ be the subspace of Hankel matrices, then the algorithm in Section \ref{secthealg} immediately yields a method for frequency estimation, (see e.g. \cite{rochberg}, \cite{actwIEEE} or \cite{acCar} for the connection between this problem and frequency estimation). We now show how the more general version in Section \ref{secgener} can be used to deal with weights or unequally spaced points. The former case is simpler so we present that separately. We also only consider the case of square matrices, although it is perfectly possible to have them non-square. In our experience, the square case works better though.

Let $\Lambda: \mathbb{C}^{2N-1} \rightarrow \mathbb{M}_{N,N}$  be defined by
\begin{equation}\label{inv}
(\Lambda a)_{j,k} = a_{j+k},
\end{equation}
i.e. the operator that takes a sequence and forms the corresponding Hankel matrix. The adjoint is then given by
$$
(\Lambda^* A)_l = \sum_{\underset{1\le j,k \le N}{j+k=l}} A_{j,k}.
$$
Set \begin{equation}\label{poi}\beta_l=\sum_{\underset{1\le j,k \le N}{j+k=l}} 1,\end{equation} and note that $\Lambda^*\Lambda=\diag(\beta)$. Let $1/\beta$ be the sequence whose $j$th element is $1/\beta_j$, and note that $\diag({1/\beta})\Lambda^*$ is the operator that projects on $\Hankelset$ and then forms the corresponding sequence $a$, i.e. it acts as a (left)-inverse of the operation \eqref{inv}. In the remainder, we will usually write simply $\frac{1}{\beta}\Lambda^*$ in place of $\diag({1/\beta})\Lambda^*$, when there is no risk of confusion. 

%
\begin{corollary} \label{corr_weight}
Let $\mu_{1},\ldots,\mu_{2N-1}$ be positive weights and assume that $q \geq \max \para{\frac{\mu}{\beta}}$ (interpreted pointwise). Set
$$
F=\Lambda\left(\frac{ f \mu}{q \beta} \right),
$$
and set
$$
\U (F,A) =F+A - \Lambda\left( \frac{\mu}{q \beta^2} \Lambda^\ast(A)  \right)
$$
Let $(W',A') = \dualproj_{\mathrm{weight}}(W,A)$ be defined by
\begin{align*}
W' &= \soc_{\runS_{\tau,q}}\Big( q \U(F,A) +\mathcal{P}_{\Hankelset^\perp}(W) \Big),\\
A' &=  \frac{1}{q-1}\left( q \U(F,A) - \mathcal{P}_{\Hankelset}(W') \right).
\end{align*}
Then, there exists $W^{\circ}$ such that $(W^{\circ},A^{\circ})$ is a fixed point of $\dualproj_{\mathrm{weight}}$ if and only if $A^{\circ}=\Lambda(a^{\circ})$ is a stationary point of the objective functional
\begin{equation}\label{equ22}
\R_{\tau}(\Lambda a) + \sum_{l} \mu_{l}\left| a_l-f_l\right|^2.
\end{equation}
The objective functional is convex if $\min \para{ \frac{\mu}{\beta} }\geq 1$, and strictly convex if $\min \para{ \frac{\mu}{\beta} }>1$.
\end{corollary}

\begin{proof}
We apply Theorem \ref{thmweighted} with $\MM=\frac{\sqrt{\mu}}{{\sqrt{q}\beta}}\Lambda^{^*}$ and $h=\frac{\sqrt{\mu}}{{\sqrt{q}}}f$. Since $\MM\Lambda a=\frac{\sqrt{\mu}}{{\sqrt{q}\beta}}\beta a=\frac{\sqrt{\mu}}{{\sqrt{q}}}a$, it is easy to see that the objective functional \eqref{equ} transforms into \eqref{equ22} when applied to the variable $A=\Lambda(a)$. Let $e_l\in\mathbb{C}^{2N-1}$ be the vector with value $\beta_l^{-1/2}$ on the $l$:th position and zeroes elsewhere, and define $E_l=\Lambda(e_l)$. Note that $(E_l)_{l=1}^{2N-1}$ is a orthonormal basis in $\Hankelset$, and that $\MM^{*}\MM$ with respect to this basis turns into a diagonal matrix whose diagonal elements are ${\frac{\mu_l}{q \beta_l}}$. By this it easily follows that the conditions on $\MM$ in Theorem \ref{thmweighted} transforms into those stated above. It also follows that
\begin{equation}\label{r5}
\MM^*\MM A=\Lambda\left( \frac{\mu}{q \beta^2} \Lambda^\ast(A)  \right),
\end{equation}
and that
$$
\MM^*h=\Lambda \frac{\sqrt{\mu}}{{\sqrt{q}\beta}}\frac{\sqrt{\mu}}{{\sqrt{q}}}f.
$$
Consequently, it holds that
$$
\MM^* (h-\MM A) + A = U(A).
$$
and therefore $\dualproj_{\mathrm{weight}}=\dualproj_{h,\tau,q}^\MM$.
\end{proof}

Given a set of sample points $X = \{x_j\}_{j=1}^J$, let $\mathcal{V}$ be the space of functions on the set $X$. We think of these as sequences on $X$, and make identification with $\mathbb{C}^{J}$. Similarly, let $Y$ be an equally spaced grid near $X$, and let functions on $Y$ be identified with $\mathbb{C}^{2N-1}$. Let ${I}_X$ be an operator that interpolates between functions on $Y$ and functions on $X$. The action of its adjoint ${I}_X^\ast$ is sometimes referred to as \emph{anterpolation}.

\begin{corollary} \label{corr_us}
Let $\mu_{1},\ldots,\mu_{J}$ be positive weights and choose $q$ such that $$\Lambda \beta^{-1} I_X^*{{\mu}}I_X \beta^{-1}  \Lambda^*\leq q I.$$
Let
$$
F=\Lambda\left(  \frac{1}{\beta} I_X^*\para{\frac{\mu f}{q}} \right),
$$
and let
$$
\U (F,A) =F+A -  \Lambda\para{ \frac{1}{q \beta} I_X^* \para{ \mu I_X \para{ \frac{1}{\beta} \Lambda^{\ast}\para{A}}} }.
$$
Let $(W',A') = \dualproj_{\mathrm{us}}(W,A)$ be defined by
\begin{align*}
W' &= \soc_{\runS_{\tau,q}}\Big( q \U(A) +\mathcal{P}_{\Hankelset^\perp}(W) \Big),\\
A' &=  \frac{1}{q-1}\left( q \U(F,A) - \mathcal{P}_{\Hankelset}(W') \right).
\end{align*}
Then, there exists $W^{\circ}$ such that $(W^{\circ},A^{\circ})$ is a fixed point of $\dualproj_{\mathrm{us}}$ if and only if $A^{\circ}=\Lambda(a^{\circ})$ is a stationary point of the objective functional
\begin{equation}\label{equ22256}
\R_{\tau}(\Lambda a) + \sum_{j} \mu_{j}\left| (I_X a-f)_j\right|^2.
\end{equation}
The objective functional is convex if $\Lambda \beta^{-1} I_X^*{{\mu}}I_X \beta^{-1}  \Lambda^*\geq I$, and strictly convex if the inequality is strict.
\end{corollary}
\begin{proof}
The proof is made analogously to the proof of Corollary \ref{corr_weight} by choosing  $h=\frac{\sqrt{\mu}}{{\sqrt{q}}}f$ and $\MM=\frac{\sqrt{\mu}}{\sqrt{q}}I_X \beta^{-1}  \Lambda^*$.
\end{proof}

\section{Multidimensional frequency estimation on general domains}\label{secmultsfe}
We now generalize the above framework to several variables. The most straightforward way to do this would be to introduce block-Hankel matrices. However, we aim at having a framework that is more flexible, and for that reason we will work with \emph{general domain Hankel matrices}. To explain the main idea, note that a Hankel matrix $A$ can be realized as the operator $\bb{\Gamma}_a:\mathbb{C}^N\rightarrow\mathbb{C}^M$ given by \begin{equation}\label{hankeldef1}(\bb{\Gamma}_a b)_m=\sum_{n=0}^{N-1}a_{m+n}b_n,\quad b\in\mathbb{C}^N.\end{equation} The bold notation is chosen for consistency with the notation used in \cite{acCar}. This formulation is suitable to generalize to the several variable setting. Let $\Upsilon$ and $\Xi$ be open bounded and connected domains in $\mathbb{R}^d$. Given sampling lengths $\bb{l}=(l_1,\ldots,l_d)$, consider the following (non-rectangular) grids $$\bb{\Upsilon}=\{\bb l\bb n=(l_1n_1,\ldots,l_d n_d)\in \Upsilon:~\bb n\in\mathbb{Z}^d\},$$ and analogously
$$\bb{\Xi}=\{\bb l\bb n\in \Xi:~\bb n\in\mathbb{Z}^d\},$$ i.e., $\bb{\Upsilon}$ and $\bb{\Xi}$ are grids covering the respective domains, and the grid lengths are implicit in the notation.
\begin{figure}
\begin{center}
\fbox{\includegraphics[width=\linewidth,trim=0in 2in 0in 1.5in]{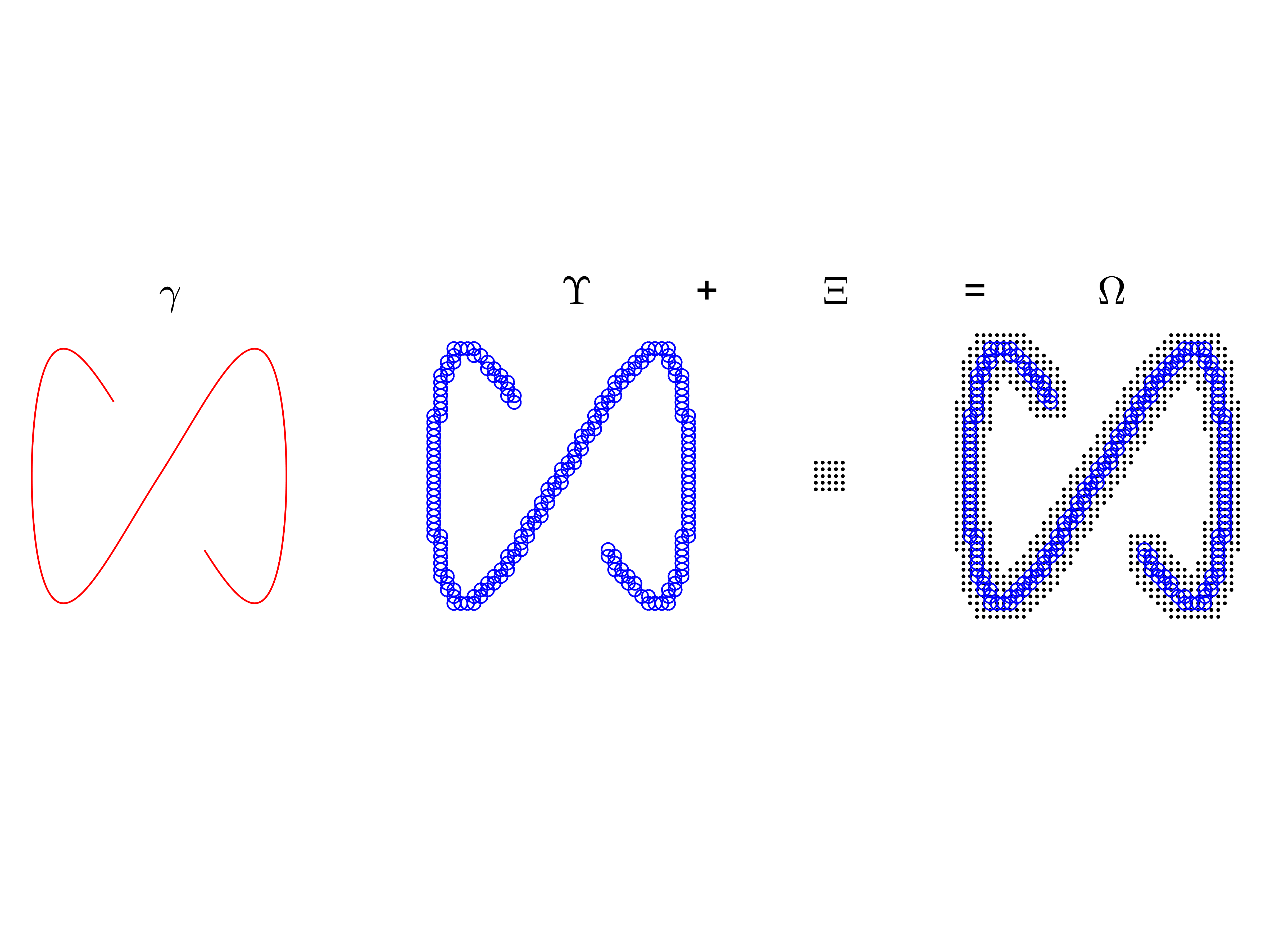}}
\end{center}
\caption{The construction of generalized multidimensional Hankel matrices. Assume that we have data sampled on a set $\gamma$, in this case a curve. The set $\bb\Upsilon$ is the set of nearest neighbor points on a equally spaced grid. The set $\bb\Xi$ is a small rectangular domain and $\bb\Omega$ is obtained as $\bb\Upsilon + \bb\Xi$. \label{fig:gen_domain} }
\end{figure}
We define $\ell^2(\bb{\Upsilon})$ to be the Hilbert space of all functions $b:\bb{\Upsilon}\rightarrow \mathbb{C}$ with the norm $$\|b\|^2=\sum_{\bb{y}\in\bb{\Upsilon}}|b(\bb{y})|^2.$$ Moreover set $$\bb{\Omega}=\bb{\Upsilon}+\bb{\Xi}.$$ Given $a\in\ell^2(\bb{\Omega})$, we define, in analogy with \eqref{hankeldef1}, the general domain Hankel operator as the summing operator $\bb\Gamma_a:\ell^2(\bb{\Upsilon})\rightarrow\ell^2(\bb\Xi)$
given by \begin{equation}\label{hankeldef2}\bb\Gamma_a b(\bb x)=\sum_{\bb{y}\in\bb{\Upsilon}} a(\bb x+\bb y)b(\bb y),\quad \bb x\in \bb \Xi.\end{equation}
We may of course represent $b$ as a vector, by ordering the entries in some (non-unique) way. More precisely, let $|\boldsymbol{\Upsilon}|$ denote the amount of elements in $\boldsymbol{\Upsilon}$, and pick any bijection \begin{equation}\label{bij}o_y:\{1,\ldots,|\boldsymbol{\Upsilon}|\}\rightarrow\boldsymbol{\Upsilon}.\end{equation} The bijection $o_y$ could for example be the lexicographical ordering of the elements in $\boldsymbol{\Upsilon}$. We can then identify $b$ with the vector $\tilde{b}$ given by $$(\tilde{b}_j)_{j=1}^{|\boldsymbol{\Upsilon}|}=b_{o_y(j)}.$$
Letting $o_x$ be an analogous bijection for $\boldsymbol{\Xi}$, it is clear that $\boldsymbol{\Gamma}_{f}$ can be represented
as a matrix, where the $(m,n)$'th element is $f(o_x(m)+o_y(n))$. If both $\Xi$ and $\Upsilon$ are squares in $\mathbb{R}^2$, these matrices are just standard block-Hankel matrices. A larger discussion of their structure is found in \cite{acCar}, see in particular Section 4.

Let $\Lambda:\ell^2(|\bb \Omega|)\rightarrow \mathbb{M}_{|\bb\Xi|,|\bb{\Upsilon}|}$ be the map that sends a function $a$ to its corresponding generalized domain Hankel matrix;$$\Lambda( a)=(a(o_x(m)+o_y(n)))_{m,n},$$ and let $\Hankelset$ be the range of this operator. Theorem \ref{buli} then applies and provides a minimizer to \eqref{Lagrange_final}. Note that
\begin{equation}\label{Lagrange_multidim}\begin{aligned}
&\argmin_{A\in \Hankelset} \R_{\tau} (A) +q \|A-F\|^2=\argmin_{a\in \ell^2(\bb\Omega)} \R_{\tau} (\Lambda(a)) +q \|\Lambda(a)-\Lambda(f)\|^2=\\&\argmin_{a\in \ell^2(\bb\Omega)} \R_{\tau} (\Lambda(a)) +q \sum_{\bb w\in\bb\Omega}\beta(\bb w)|a(\bb{w})-f(\bb{w})|^2,\end{aligned}
\end{equation}
where the weight-function $\beta$ evaluated at a point $\bb w$ simply is the amount of times that the value $a(\bb w)$ appears in the matrix $A=\Lambda(a)$, in analogy with \eqref{poi}.

In Figure \ref{fig:gen_domain} we illustrate how a general domain Hankel matrix could be constructed. Suppose that we have data available along some subset of $\mathbb{R}^2$, in this case along a line $\gamma$. We want to inscribe the curve $\gamma$ in a subset of an equally spaced grid $\bb\Omega$. We would like to factorize the set $\bb\Omega$ as above. The generating values for the general domain Hankel matrix will in this case ``live close'' to where we have data. This is of particular importance if the exponential functions that we want to recover have frequencies with non-zero real part, as they easily can become huge far away from where the samples are available. The shape of the domain $\bb\Xi$ will control how many exponential functions that can be recovered. We refer to \cite{andersson2010nonlinear_exp,andersson2010nonlinear_wp} for details about how to recover the exponentials and the relationship between the size of $\bb\Xi$ and the number of exponential functions that can be recovered.

\section{Numerical simulations}\label{secnum}
\subsection{Equally spaced data}
We begin with studying the case of equally spaced samples without weights. For this case the standard methods for frequency estimation are applicable. We will conduct some comparison of the results obtained by the proposed method against the results obtained by the ESPRIT method \cite{ESPRIT}. One difference between the proposed method and the ESPRIT method is that the number of exponentials sought for is predetermined in ESPRIT, while in the proposed method a penalty level is used instead of prescribing the number of exponentials. By modifying the penalty level $\tau$, i.e., by modifying $\runS_{\tau}$ in (\ref{runS}), we can adapt the fixed point algorithms for making approximations using a fixed number of exponential functions. The simplest way to do this is to choose $\tau$ such that $\tau= q \sigma_K(W)$, if $K$ is the number of exponentials sought for. Note that the fixed point operator is then modified in between iterations, and that the convergence results are not directly applicable.

To avoid this difficulty, we will conduct experiments on sums of exponentials of the form (\ref{Hankel_ell2_org}) with $|c_k|=1$. If the data is contaminated by noise, then the number of exponentials found can easily vary  form small signal to noise ratios, if the strength (i.e. $|c_k|$) of the different exponentials have large variations.

We conduct experiments on a function of the form (\ref{Hankel_ell2_org}) with frequencies $\zeta_k$ and coefficients $c_k$ given in Table \ref{tab:4exp} for $|x|\le \frac{1}{2}$.
\begin{table}[h]
\centering
\caption{Frequencies and coefficients}
\label{tab:4exp}
\begin{tabular}{ | l | l |  l |  l | l|}
\hline
$\zeta_k$ & -23.141i & -3.1416i         & 2.7183i           & 31.006i           \\
\hline
$c_k$   & 1.00000 & 0.62348+0.78183i & -0.22252+0.97493i & -0.90097+0.43388i\\
\hline
\end{tabular}
\end{table}

\begin{figure}
\begin{center}
\includegraphics[width=0.48\linewidth,trim=0in 0in 0in 0in]{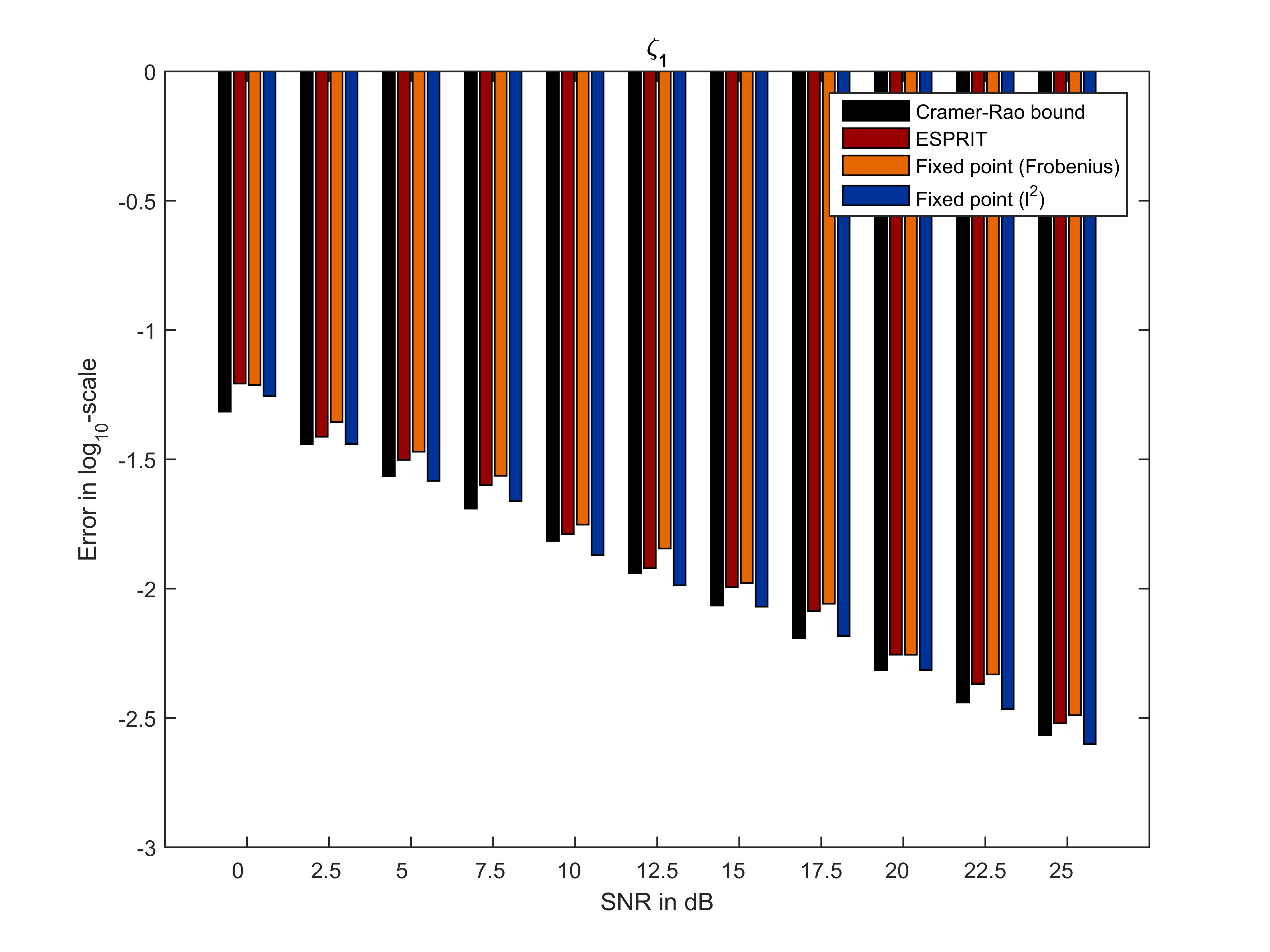}
\includegraphics[width=0.48\linewidth,trim=0in 0in 0in 0in]{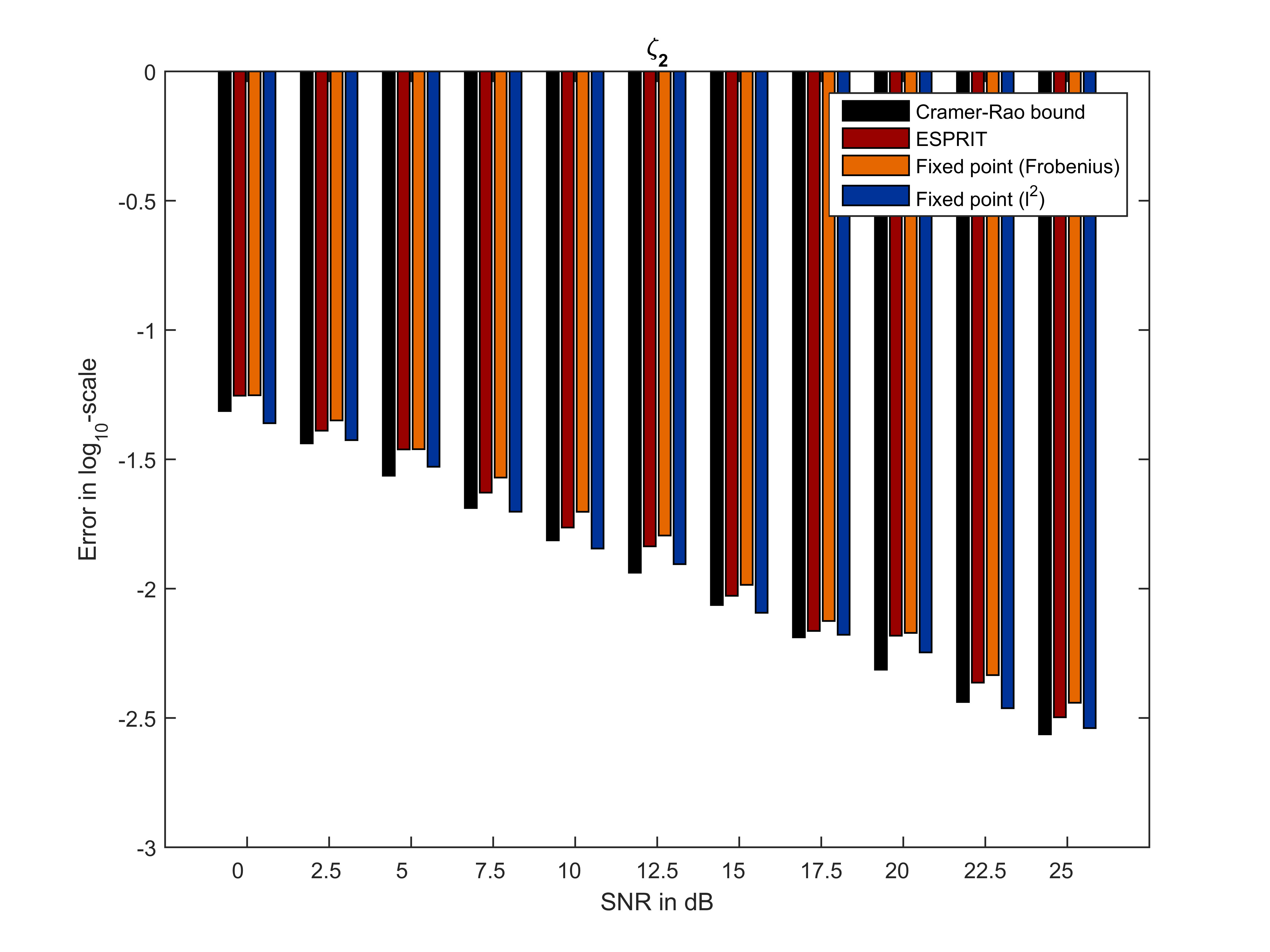}
\includegraphics[width=0.48\linewidth,trim=0in 0in 0in 0in]{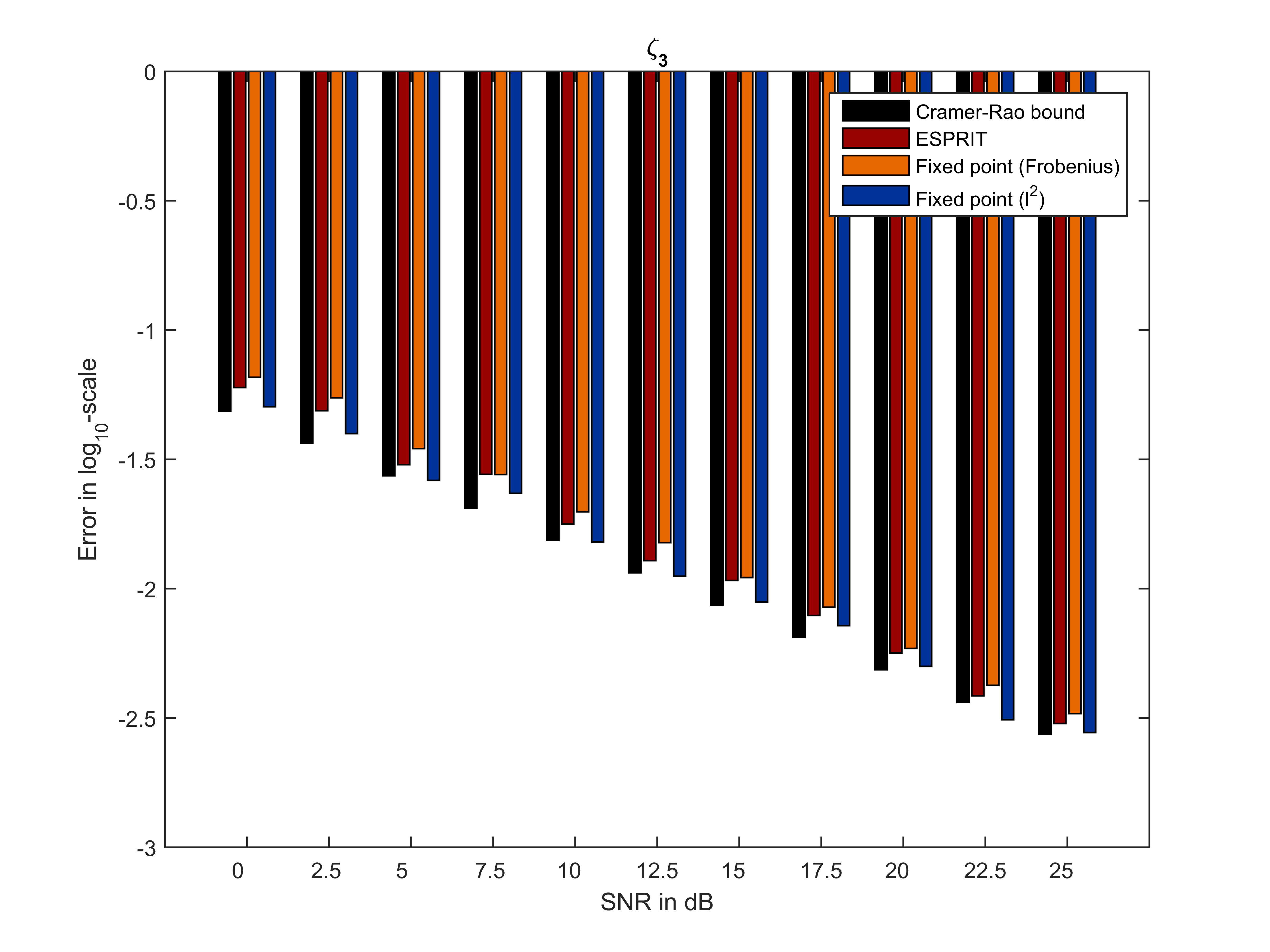}
\includegraphics[width=0.48\linewidth,trim=0in 0in 0in 0in]{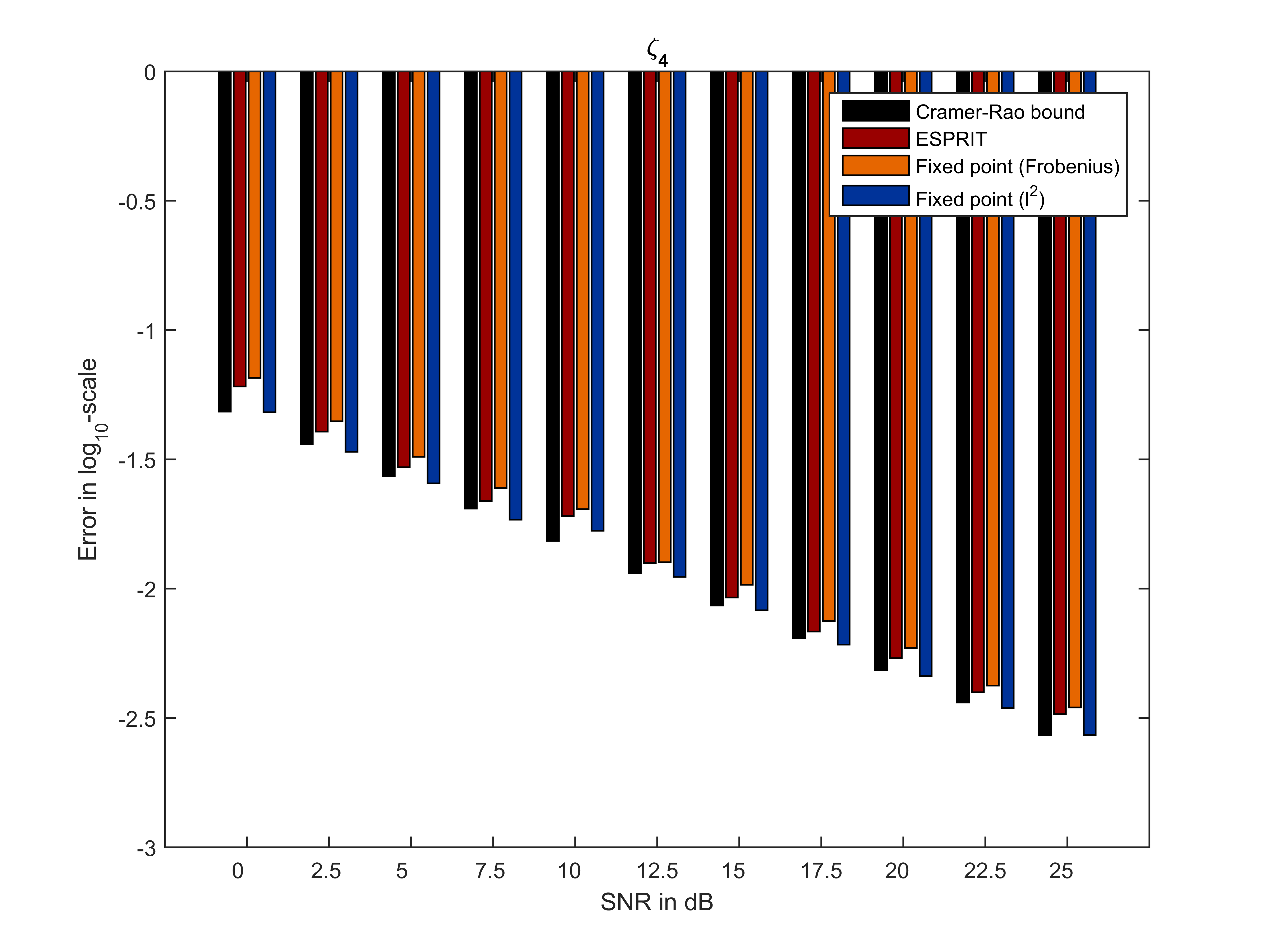}
\end{center}
\caption{Errors in the estimation of $\Re{\zeta_k}$, $k=1, \dots, 4$, for equally spaced sampling using 257 nodes: Cramer-Raó bounds, ESPRIT error, Fixed point-Frobenius, Fixed point-$\ell^2$.\label{CRB}}
\end{figure}
The function is sampled at 257 points, and white noise is added with different signal-to-noise rations (SNR). We test the fixed point algorithm given in Theorem \ref{buli} for a Frobenius norm error, and using the fixed point algorithm given in Corollary \ref{corr_weight}, for the standard $\ell^2$ norm. In Figure \ref{CRB} we show the errors in recovering the frequencies from Table \ref{tab:4exp} for the two fixed point algorithms, along with the errors obtained from ESPRIT and the Cramér-Rao bound for the estimates. We can see that in comparison to the Frobenius based fixed point algorithm, the ESPRIT method works better in estimating the frequencies, but also that the $\ell^2$ version works better than ESPRIT. For this function, all three methods work well, and the obtained errors are close to the Cramér-Rao bound. Note that the methods are trying to optimize the approximation error of the function using few exponentials rather than estimating the frequencies present in the signal.

In Figure \ref{fig:frol2err} we compare the approximation errors for the two methods against that or ESPRIT. The errors depicted are
$$
\| A_{\mathrm{ESPRIT}} - F\|_{F} -\| A_{\mathrm{Fro}} - F \|_{F},
$$
and
$$
\| a_{\mathrm{ESPRIT}} - f\|_{\ell^2} -\| a_{\ell^2} - f\|_{\ell^2},
$$
respectively, where $A_{\mathrm{Fro}}$  denotes the outcome from the  fixed-point algorithm in Theorem \ref{buli}, and $a_{\ell^2}$ the output of the fixed-point algorithm in Corollary \ref{corr_weight}. The errors are computed for the 11 different SNR values of Figure \ref{CRB} using 100 simulation for each value. The errors are then normalized in relation to the SNR for the sake of visualizing them simultaneously. The scale on the $y$-axis is thus a relative scale, and we have rescaled it so that the largest difference between the errors produced by the different methods is one. We can see that for both cases the errors are positive, meaning that the approximations obtained by the two fixed point methods consequently produce approximations with smaller error in the respective norm.
\begin{figure}
\begin{center}
\includegraphics[width=0.48\linewidth,trim=0in 0in 0in 0in]{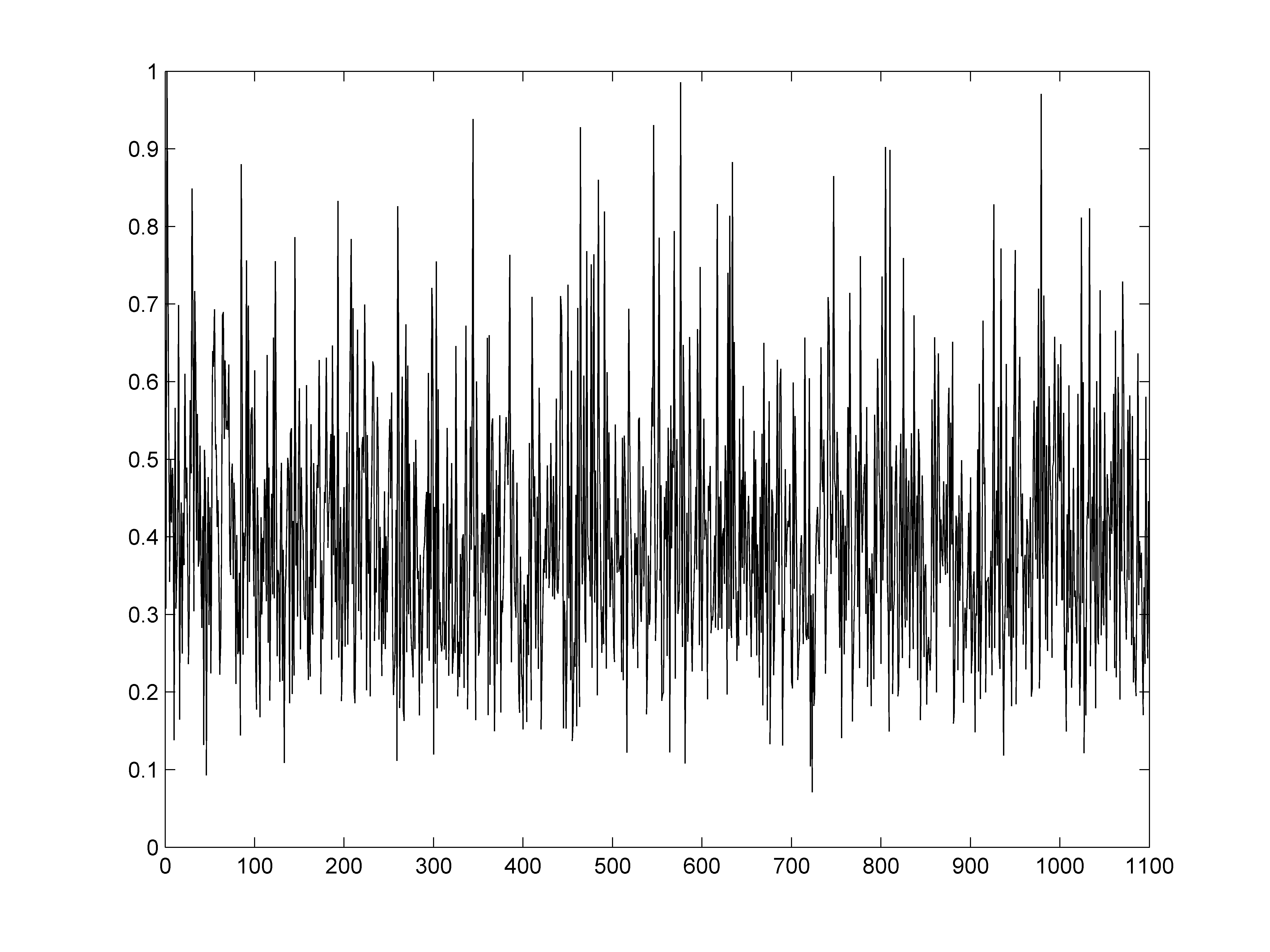}
\includegraphics[width=0.48\linewidth,trim=0in 0in 0in 0in]{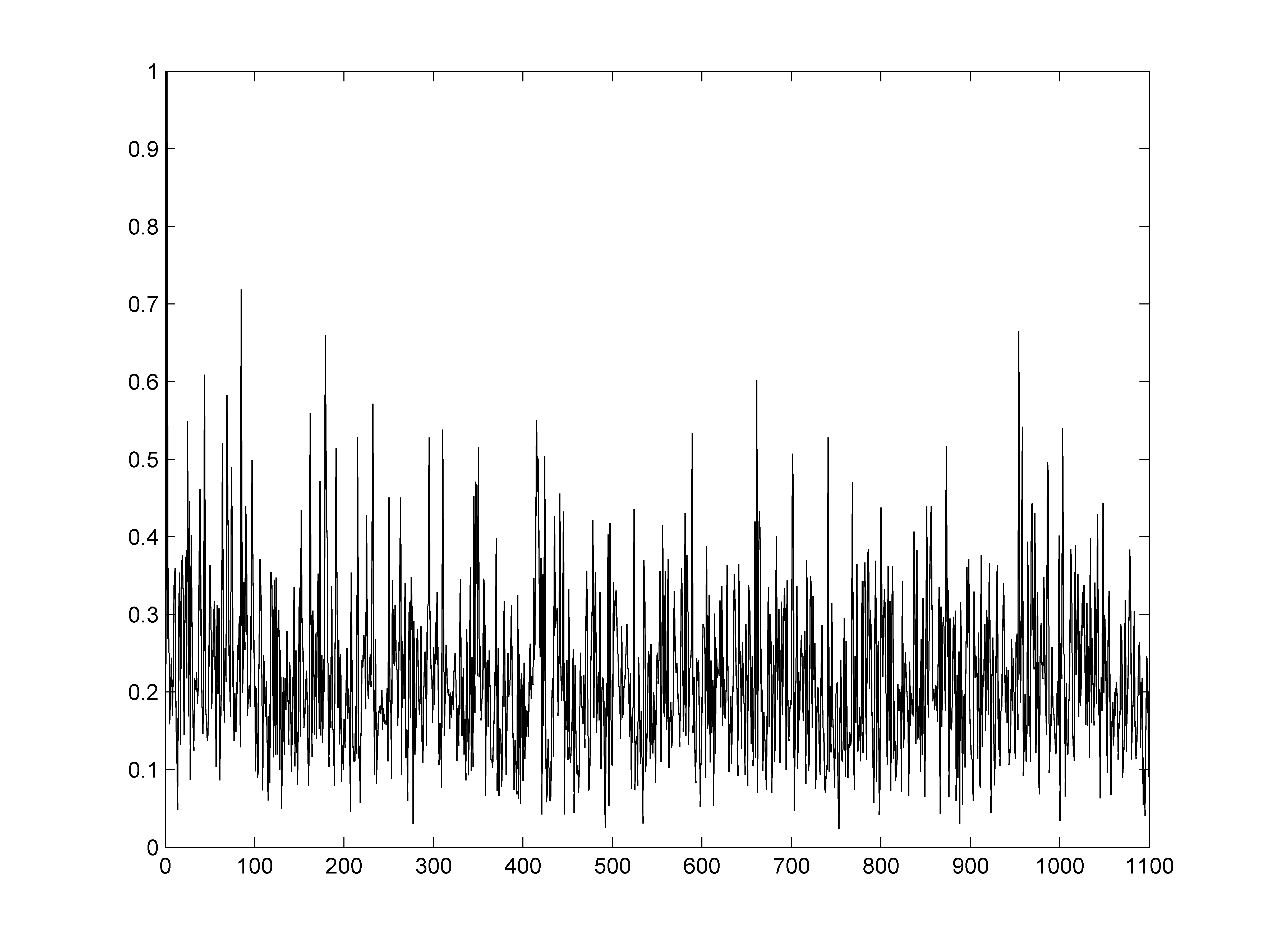}
\end{center}
\caption{Errors in the estimation of $\Re{\zeta_k}$ for equally spaced sampling using 257 nodes: Left comparison of the error produced by ESPRIT and the fixed-point algorithm given in Theorem \ref{buli}; Right comparison of the error produced by ESPRIT and the fixed-point algorithm in Corollary \ref{corr_weight} \label{fig:frol2err}}.
\end{figure}

\subsection{Missing data}
\begin{figure}
\begin{center}
\includegraphics[width=0.45\linewidth,trim=1.3in .4in .4in .4in]{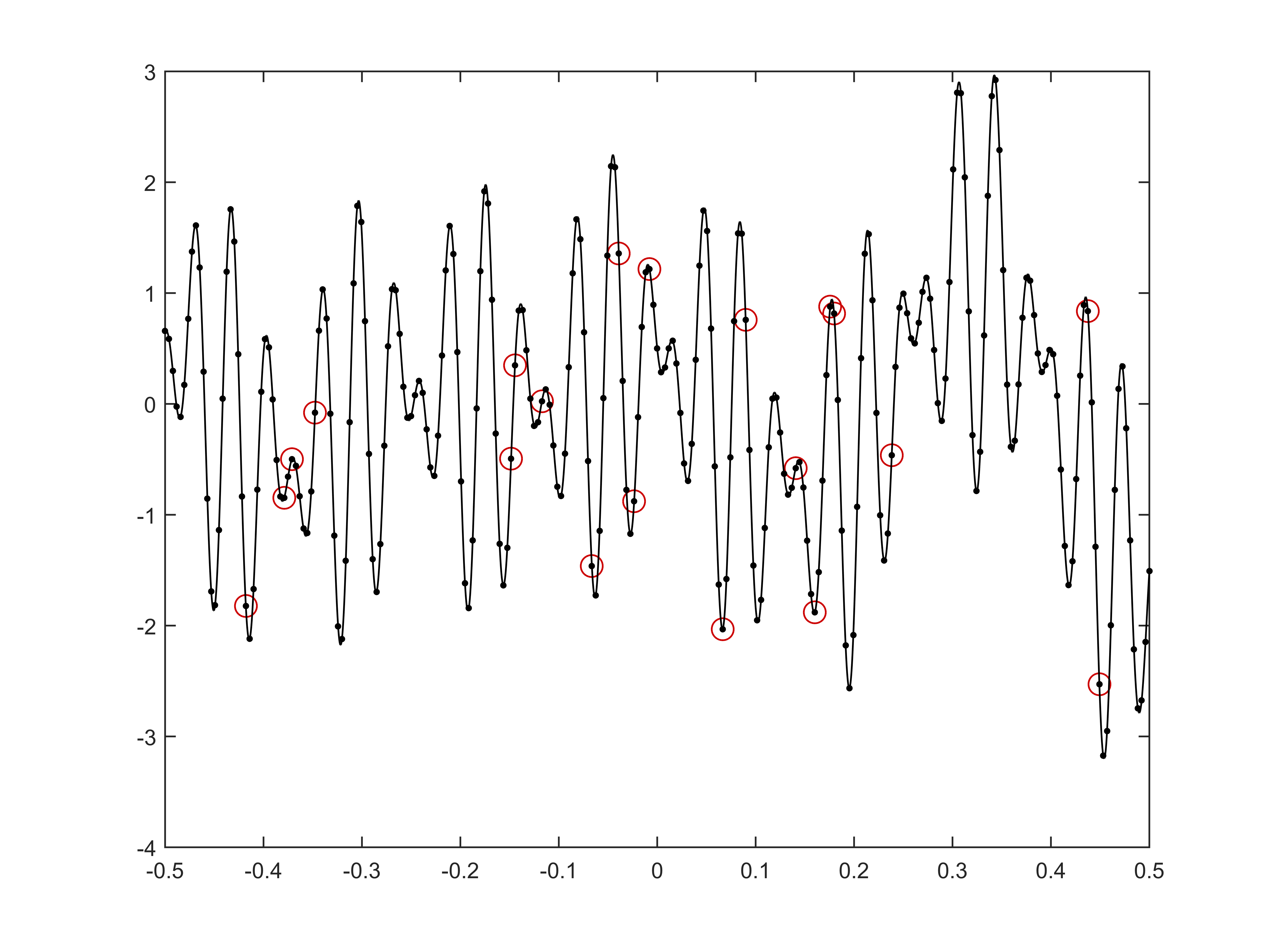}
\includegraphics[width=0.45\linewidth,trim=.4in .4in 1.3in .4in]{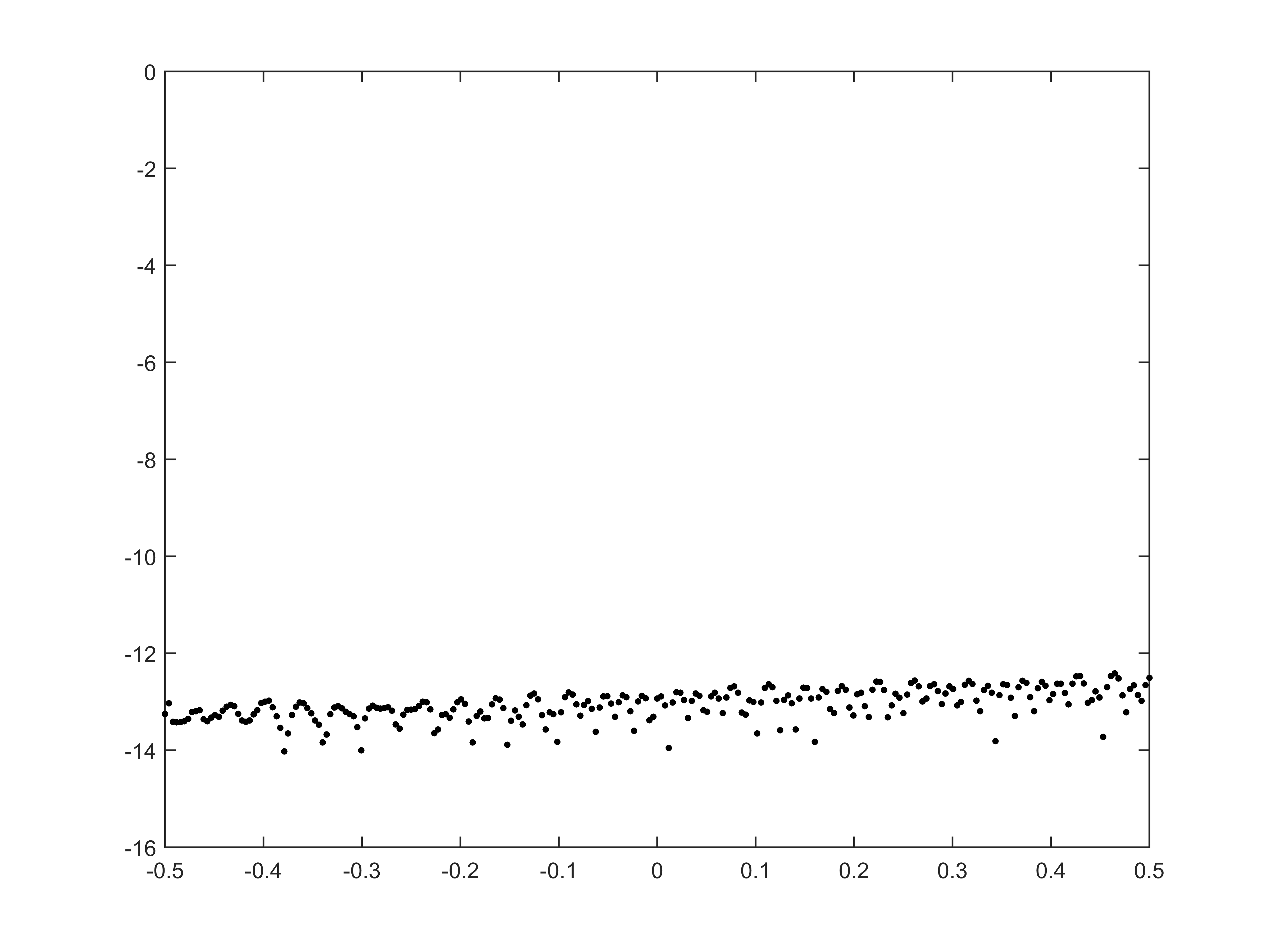}
\end{center}
\caption{Missing data problem. The right panel shows the true signal in black, with equally spaced samples indicated by black dots, and the available samples indicated by red circles. The right panel shows the error between the original function and the reconstructed function in $\log_{10}$-scale\label{fig:missing_data_perfect}}
\end{figure}
In this subsection we study the performance of the proposed method in the case of missing data. We assume that the samples are known only at a subset of an equally spaced sampling in one dimension. The function we study is the same as before, i.e. composed from the coefficients and frequencies given in Table \ref{tab:4exp}.  In the first test, we try to recover the function in the absence of noise, but with only few measurements available. The function is originally sampled at 257 points at one the interval $[-0.5, 0.5]$, and we assume that only the 20 samples with indices $$ \{22, 32,34, 40, 91, 92, 99, 112, 119, 123, 127, 146, 152, 165, 170, 174, 175, 190, 241, 244 \} $$ are known (these are randomly chosen). The real part of the original function is shown in black in the left panel of Figure \ref{fig:missing_data_perfect}, and the available data is illustrate by red circles. In the left panel the point-wise error between the reconstruction and the original function is shown in logarithmic scale. We can see that the reconstruction is accurate up to machine precision. Note that since $\mu_j=0$, this is not guaranteed from Corollary \ref{corr_weight}. However, comparatively few samples are typically needed for the method to achieve perfect recovery in the absence of noise.

\begin{figure}
\begin{center}
\includegraphics[width=0.45\linewidth,trim=0in 0in 0in 0in]{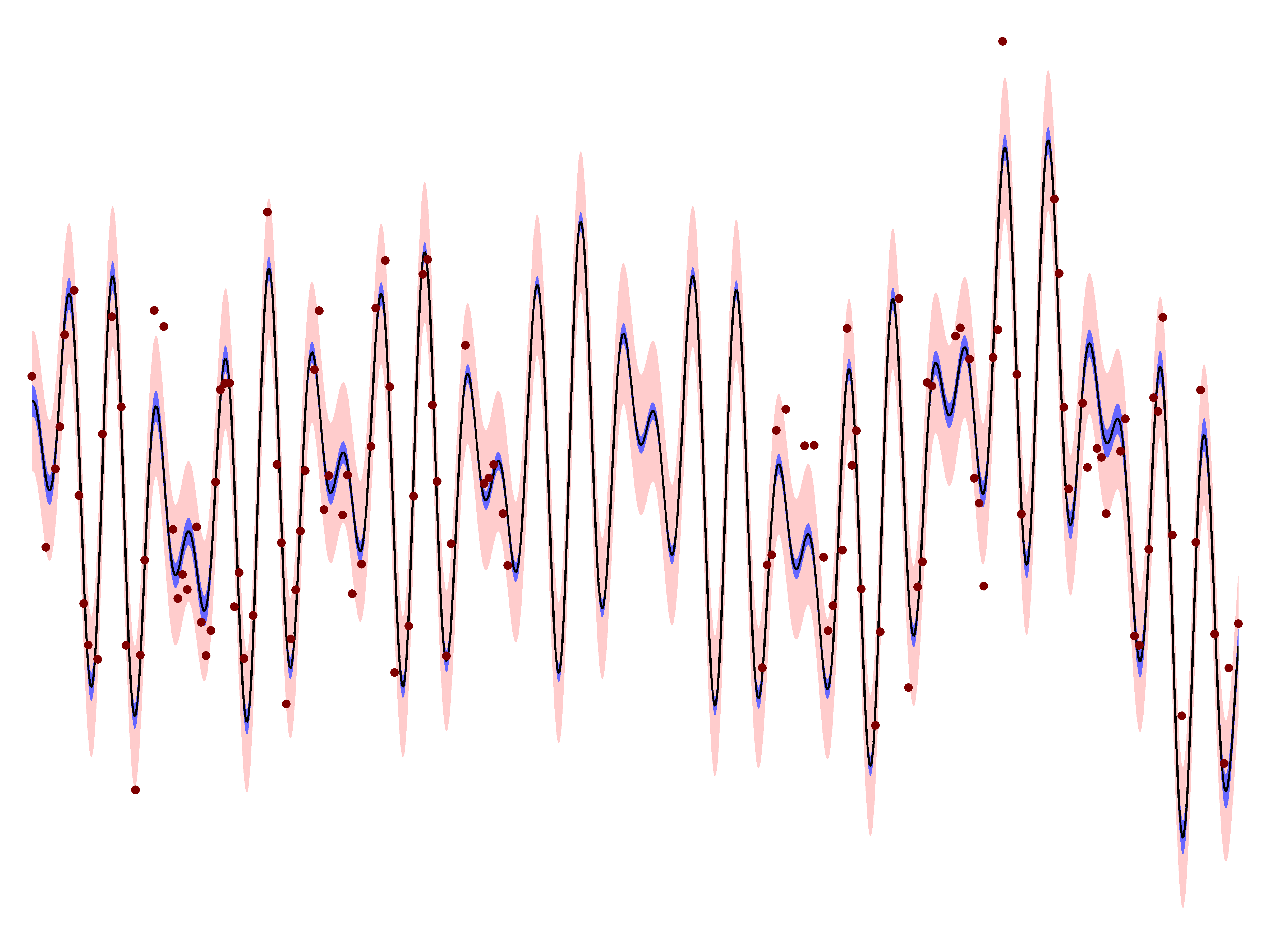}
\includegraphics[width=0.45\linewidth,trim=0in 0in 0in 0in]{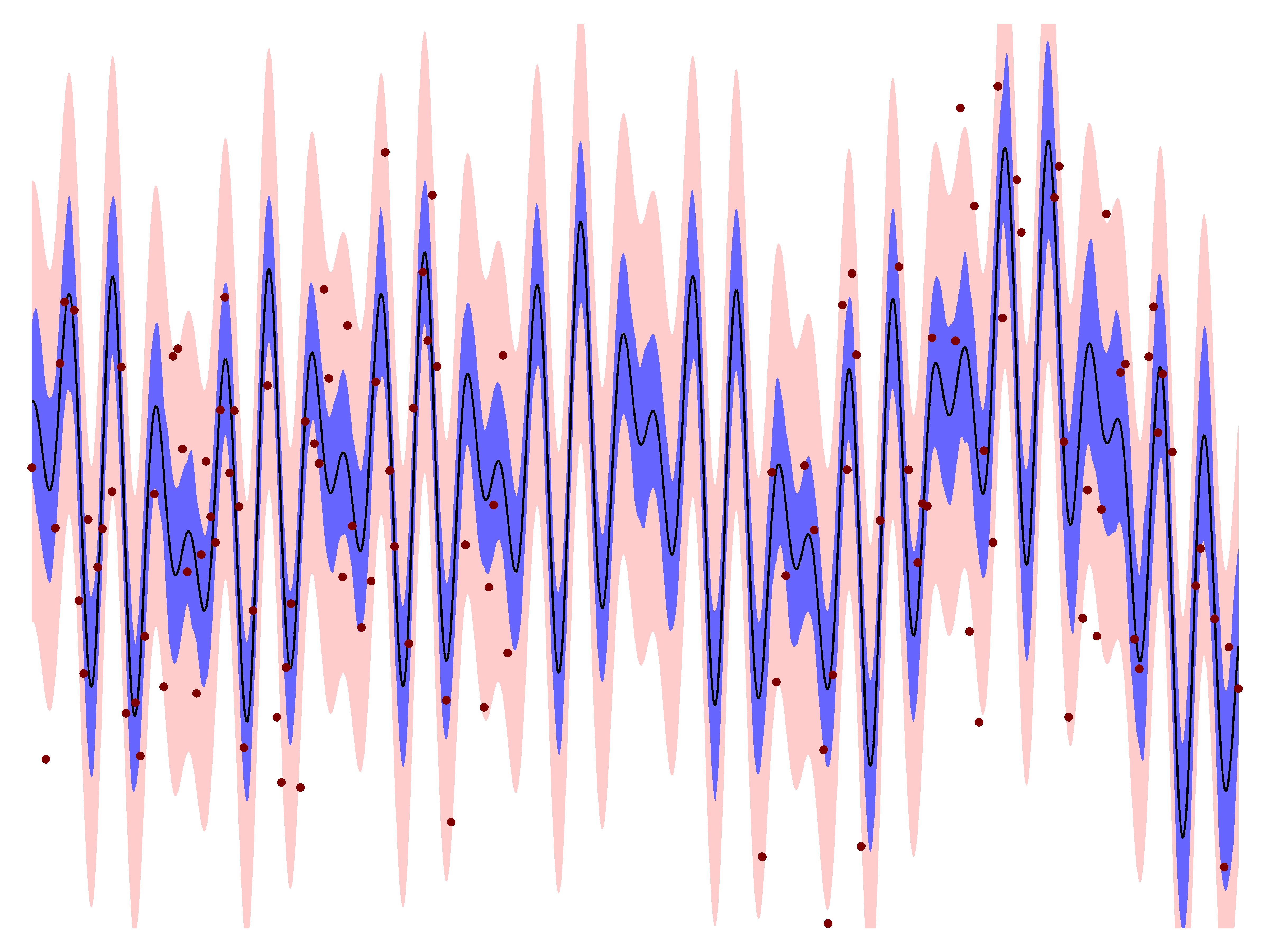}
\end{center}
\caption{Missing data problem. The right panel shows simulations for SNR=10 and the right panel for SNR=0. The black curve indicate the true values of the function; the red dots show the signal with noise at available locations; the blue areas illustrate the standard deviation of point-wise reconstructions for 100 simulations; and the pink areas illustrate the standard deviation of the added noise realizations \label{fig:missing_data}}
\end{figure}
Next, we consider the same function but now noise is added. Half of the samples are kept (129 points). Moreover, the sampling points are chosen so that there is a gap on the interval $[-0.1,0.1]$.
  We conducted two sets of simulations, one for SNR=10 and one for SNR=0, and for both noise levels we ran the simulations 100 times. In Figure \ref{fig:missing_data} we illustrate the results. The original function is displayed by the black curve, and for each of the figures, the red dots indicate the noise realization for one simulation. Note that the samples (red dots) deviate quite a lot from the true data. The pink shaded area illustrate the standard deviations for the the noise, and the blue shaded area illustrate the standard deviation of the errors from the obtained approximations. We can see that the errors in the reconstructions are substantially smaller than the original noise level. A similar experiment can also be done using the unequally spaced sampling formulation dealt with in Corollary \ref{corr_us}.

\subsection{Two-dimensional recovery along a curve}
\begin{figure}
\begin{center}
\includegraphics[width=0.8\linewidth,trim=0in .9in 0in 0.4in]{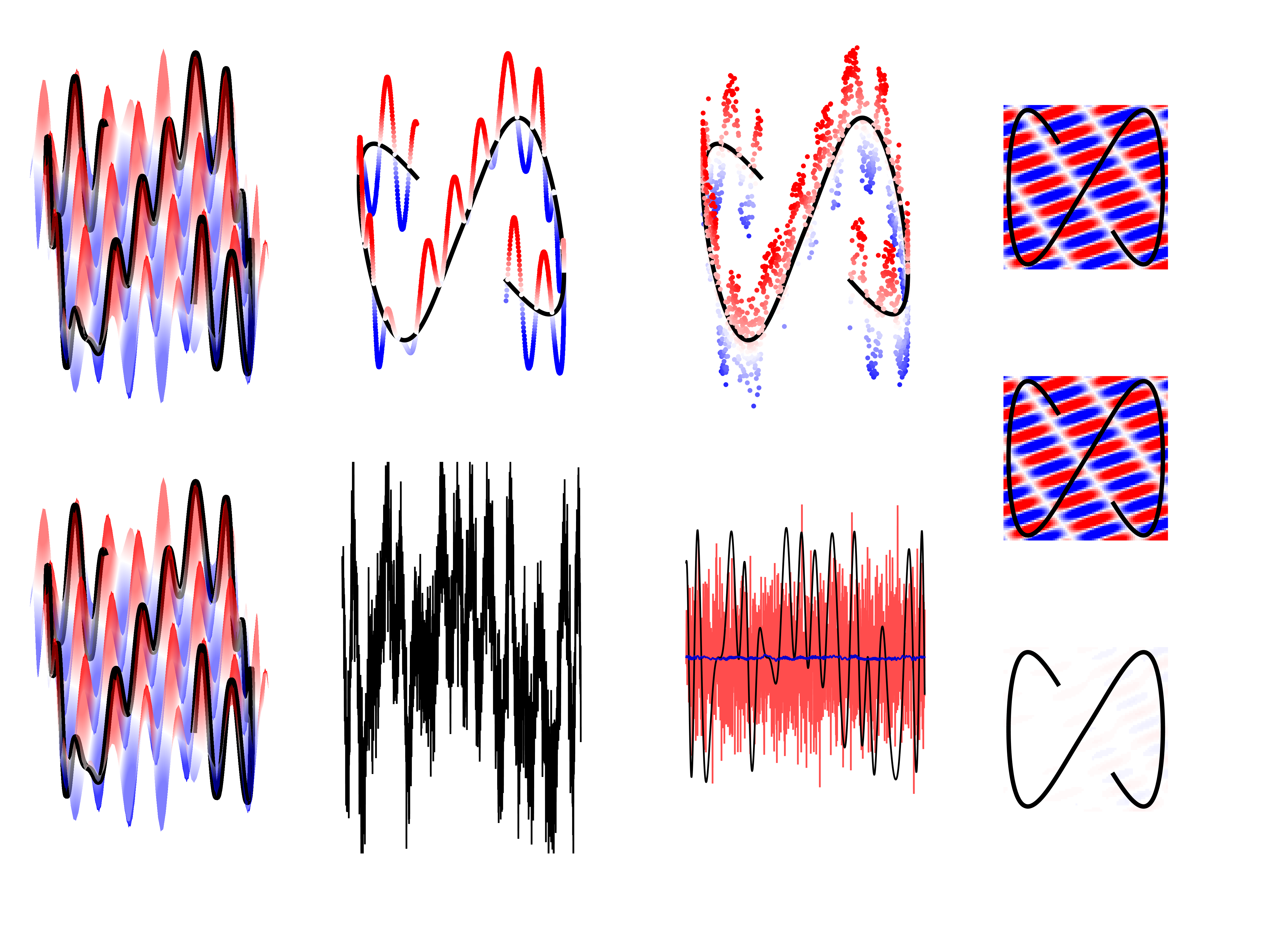}
\end{center}
\caption{Approximation of a two-dimensional function sampled along a curve. The underlying function is a sum of two exponentials. All illustrations show only the real parts.  The leftmost top panel shows the underlying sum of two exponential functions in blue and red, and the sampling curve is shown in black; the second top panel shows the sampling curve in black with the sampled values overlayed in blue and red; the third top panel left show the noisy measurement; The leftmost bottom panel shows the reconstructed data corresponding the the panel right above it; the second bottom panel depicts sampled noisy data along the curve; the third bottom panel shows the noise-free data along the sampling curve in black, the reconstruction error  (along the sampling curve) in blue and the noise is depicted in red. The rightmost column shows from top to bottom: top view of the true two-dimensional sum of two exponentials; the reconstructed sum; the corresponding reconstruction error. \label{fig:2dcurve_2}}
\end{figure}

Next, we consider the case of unequally spaced sampling in two dimensions. In particular, we will choose points densely sampled along a curve. There are no particular constraints of the method that limits it to such sampling, it is simply an interesting and comparatively difficult problem. We will use the same curve as the one ($\gamma$) shown in Figure \ref{fig:gen_domain}.
Due to the problems of satisfying the convexity constraints in Corollary \ref{corr_us}, the algorithm can converge to for instance a zero solution. To avoid problems of this type, we assume that the number of exponential functions $K$ is known in advance and choose $\tau =\sigma_K(W)/q$ at each iteration.

In a first example we try out the proposed method in the case with $K=2$ exponential functions in two variables.  The underlying function is shown at the top left panel of Figure \ref{fig:2dcurve_2}. The black curve indicates the curve along which sampling is conducted. The sampling setup is also illustrated in the second (top) panel. Here the sampling curve ($\gamma$) is indicated in black, and the sampled data values along the curve are illustrated by red, white and blue (densely sampled) dots. We perturb data by noise. For this case we apply a rather high level of noise, in particular Gaussian white noise with SNR=1 dB. The noisy measurements are illustrated in the third (top) panel from the left. We now apply the fixed-point algorithm. Upon convergence, we use the singular vectors on the (square) $\bb\Xi$-domain to obtain frequency estimates following the strategy described in \cite{andersson2010nonlinear_exp}. The corresponding coefficients can then be found by the least squares method. Once the frequencies and coefficients have been retrieved, we can reconstruct the underlying two-dimensional function. The result is shown in the bottom left panel. The second bottom panel show the measured (noisy) signal (along $\gamma$); The third (bottom) panel show the noise-free signal in black; the noise itself in red; and the residual between the noise-free and the reconstructed signals in blue. The right column of panels of Figure \ref{fig:2dcurve_2} show in top views from top to bottom: the noise-free data; the reconstruction; and the corresponding residual. We can see that a rather good reconstruction was obtained, despite the high noise-level. This is partially because the original data set consisted of only two exponential functions.
\begin{figure}
\begin{center}
\includegraphics[width=0.8\linewidth,trim=0in 0in 0in 0in]{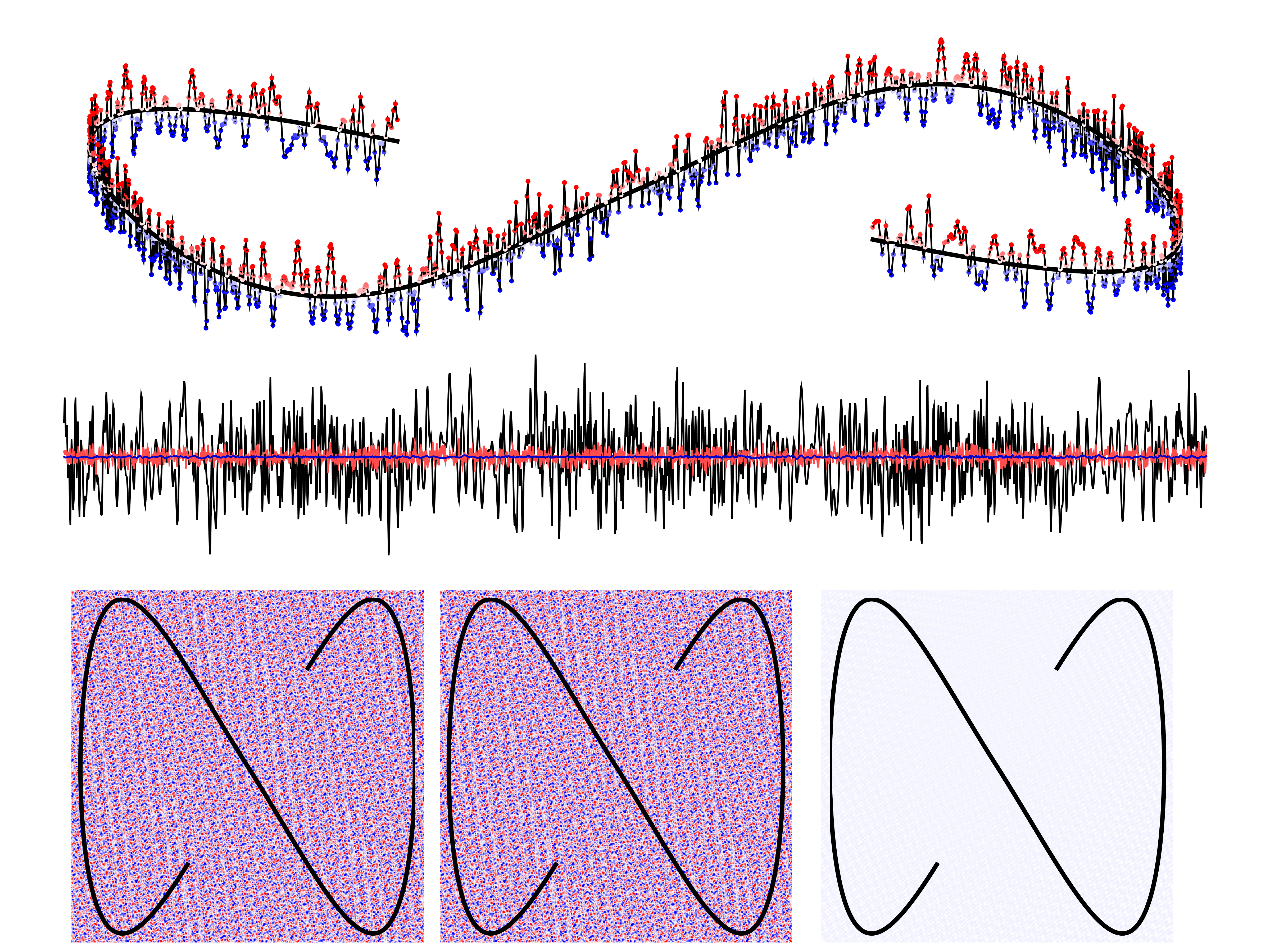}
\end{center}
\caption{Approximation of a two-dimensional function sampled along a curve. The underlying function is a sum of ten exponentials. The top panel shows the sampled data along a planar curve. The middle panel show from left to right: The true data in black, the noise in red and the reconstruction error in blue (along the curve). The bottom panel shows; the original sum of exponentials on a square; the reconstructed sum of exponentials; and reconstruction error \label{fig:2dcurve_10}}
\end{figure}

We now try the same problem but instead with $K=10$ exponential functions and an SNR of 5 dB. The top panel show the data sampled along the curve $\gamma$ as before. We can see that the data looks substantially more complicated in this case. The same procedure as for the previous example is now applied. The middle panel show the noise-free data along $\gamma$ in black and the corresponding residual with regards to the reconstruction in blue. The noise is shown in red. The three bottom panels show top views of the noise-free data; the reconstruction; and the corresponding residual for the underlying two-dimensional function. Finally, we provide the original and estimated frequencies for this two-dimensional case in Table \ref{tab:2dcurve_freq}.
\begin{table}[h]
\centering
\caption{Original and estimated complex frequencies. \label{tab:2dcurve_freq}}
\label{tab:2d}
\resizebox{\textwidth}{!}{\texttt{
\begin{tabular}{| l | l | l | l |}
\hline
$\zeta_1$                              & $\zeta_2$                              & $\zeta_1^\mathrm{est}$   & $\zeta_2^ \mathrm{est}$   \\\hline
-50.0000 + 0.i &-50.0000 + 0.i &-50.0031 + 0.00265i &	-49.9989 + 0.00255i \\ \hline
-42.5903 + 0.i &+12.8766 + 0.i &-42.5882 + 0.00118i &	+12.8765 - 0.00100i \\ \hline
-35.1806 + 0.i &+40.7788 + 0.i &-35.1817 + 0.00142i &	+40.7763 - 0.00027i \\ \hline
-20.3613 + 0.i &-2.72062 + 0.i &-20.3618 + 0.00278i &	-2.72020 - 0.00275i \\ \hline
-6.67724 + 0.i &-34.2402 + 0.i &-6.67772 - 0.00267i &	-34.2389 + 0.00089i \\ \hline
+4.98657 + 0.i &-42.4351 + 0.i& +4.98607 - 0.00091i &	-42.4358 + 0.00056i \\ \hline
+9.27734 + 0.i &+18.0841 + 0.i &+9.27612 - 0.00129i	 &       +18.0857 - 0.00397i \\ \hline
+19.4458 + 0.i &-9.77757 + 0.i &+19.4469 + 0.00072i &	-9.77517 - 0.00032i \\ \hline
+36.6455 + 0.i &-27.3052 + 0.i & +36.6436 + 0.00016i &	-27.3033 + 0.00502i \\ \hline
+47.9858 + 0.i &+33.6813 + 0.i & +47.9836 + 0.00042i &	+33.6802 + 0.00293i \\ \hline
%
\end{tabular}
}}
\end{table}

\section{Conclusions}\label{secconcl}
We have presented fixed-point algorithms for the approximation of functions by sparse sums of exponentials. These algorithms hold for the general class of approximation problems with a rank and subspace constraint on the solution. The proposed method can deal with the standard frequency estimation problem (written in terms of Hankel matrices), but it also extends to more general cases, for instance in the presence of weights and unequally spaced sampling. We provide theorems about convergence, and also a condition on if the optimal solution was obtained. We also describe how to apply the results for the approximation using sparse sums of exponentials in several variables. The formulation allows for approximations and estimation in general sampling geometries, and we illustrate this performing two-dimensional frequency estimation given data sampled along a curve.


\bibliographystyle{plain}
\bibliography{referenser}

\end{document}